\documentclass[twoside,11pt,reqno]{amsart}
\usepackage{amsmath,amssymb,amscd,mathrsfs,amscd, todonotes,appendix}
\usepackage{graphics,verbatim}
\usepackage{todonotes}
\usepackage{enumitem}
\usepackage{hyperref}
\usepackage{setspace} 

\usepackage{xcolor}

\newcommand{\losemi}{{\otimes \kern -.78em \ltimes}}
\newcommand{\rosemi}{{\otimes \kern -.78em \rtimes}}

\newcommand{\Hom}{\ensuremath{\operatorname{Hom}}}

\newcommand{\Ext}{\operatorname{Ext}}

\newcommand{\ind}{\operatorname{ind}}

\renewcommand{\a}{\alpha}

\newcommand{\ep}{\epsilon}

\newcommand{\ga}{\gamma}
\newcommand{\s}{\sigma}
\newcommand{\om}{\omega}

\newcommand{\la}{\lambda}

\newcommand{\St}{\operatorname{St}}

\newcommand{\hQ}{\widehat{Q}}
\newcommand{\hZ}{\widehat{Z}}
\newcommand{\hL}{\widehat{L}}

\makeatletter
\newcommand{\leqnomode}{\tagsleft@true}
\newcommand{\reqnomode}{\tagsleft@false}
\makeatother

\newtheorem{theorem}{Theorem}[subsection]

\makeatletter\let\c@fact\c@theorem\makeatother

\makeatletter\let\c@note\c@theorem\makeatother

\newtheorem{lemma}{Lemma}[subsection]
\makeatletter\let\c@lemma\c@theorem\makeatother

\makeatletter\let\c@lemma\c@theorem\makeatother

\newtheorem{quest}{Question}[subsection]
\makeatletter\let\c@quest\c@theorem\makeatother

\newtheorem{prop}{Proposition}[subsection]
\makeatletter\let\c@prop\c@theorem\makeatother

\newtheorem{conj}{Conjecture}[subsection]
\makeatletter\let\c@conj\c@theorem\makeatother

\makeatletter\let\c@cor\c@theorem\makeatother

\makeatletter\let\c@defn\c@theorem\makeatother

\theoremstyle{definition}

\newtheorem{remark}{Remark}[subsection]
\makeatletter\let\c@remark\c@theorem\makeatother

\makeatletter\let\c@example\c@theorem\makeatother
\numberwithin{equation}{subsection}

\usepackage[capitalise]{cleveref}
\crefname{theorem}{Theorem}{Theorems}
\crefname{fact}{Fact}{Facts}
\crefname{note}{Note}{Notes}
\crefname{lemma}{Lemma}{Lemmas}
\crefname{alg}{Algorithm}{Algorithms}
\crefname{remark}{Remark}{Remarks}
\crefname{example}{Example}{Examples}
\crefname{prop}{Proposition}{Propositions}
\crefname{conj}{Conjecture}{Conjectures}
\crefname{cor}{Corollary}{Corollaries}
\crefname{defn}{Definition}{Definitions}
\crefname{equation}{\!\!}{\!\!} 

\newcounter{listequation}

\begin{document}

\title{On Donkin's Tilting Module Conjecture II: Counterexamples}

\begin{abstract} In this paper we produce infinite families of counterexamples to Jantzen's question posed in 1980 on the existence of Weyl $p$-filtrations for Weyl modules for an algebraic group and 
Donkin's Tilting Module Conjecture formulated in 1990. New techniques to exhibit explicit examples are provided along with methods to produce counterexamples in large rank from 
counterexamples in small rank. Counterexamples can be produced via our methods for all groups other than when the root system is of type $\rm{A}_{n}$ or $\rm{B}_{2}$. 
\end{abstract}

\author{\sc Christopher P. Bendel}
\address
{Department of Mathematics, Statistics and Computer Science\\
University of
Wisconsin-Stout \\
Menomonie\\ WI~54751, USA}
\thanks{Research of the first author was supported in part by Simons Foundation Collaboration Grant 317062}
\email{bendelc@uwstout.edu}

\author{\sc Daniel K. Nakano}
\address
{Department of Mathematics\\ University of Georgia \\
Athens\\ GA~30602, USA}
\thanks{Research of the second author was supported in part by
NSF grants DMS-1701768 and DMS-2101941}
\email{nakano@math.uga.edu}

\author{\sc Cornelius Pillen}
\address{Department of Mathematics and Statistics \\ University
of South
Alabama\\
Mobile\\ AL~36688, USA}
\thanks{Research of the third author was supported in part by Simons Foundation Collaboration Grant 245236}
\email{pillen@southalabama.edu}

\author{Paul Sobaje}
\address{Department of Mathematical Sciences \\
          Georgia Southern University\\
          Statesboro, GA~30458, USA}
\email{psobaje@georgiasouthern.edu}

\maketitle
\section{Introduction}

\subsection{} Let $G$ be a connected reductive group over an algebraically closed field $k$ of characteristic $p>0$. A fundamental problem in the representation theory of $G$ is to understand the structure of Weyl modules or equivalently the structure of the induced modules $\nabla(\lambda)$ for $\lambda\in X^{+}$ a dominant integral weight. The determination of the composition factors of $\nabla(\lambda)$ for all $\lambda\in X^{+}$ is equivalent to knowing the characters of simple modules for $G$. This central problem still remains open even though there has been much work done to formulate new character formulas using tilting modules and $p$-Kazhdan-Lusztig polynomials. In order to enhance our understanding of the structure of $\nabla(\lambda)$, Jantzen asked the question in 1980 as to whether the induced modules $\nabla(\lambda)$ have good $p$-filtrations. 
Parshall and Scott \cite{PS} proved that this holds as long as $p\geq 2(h-1)$ and the Lusztig Character Formula holds for $G$. Andersen \cite{And19} recently proved that $\nabla(\lambda)$ has a good $p$-filtration as long as $p\geq h(h-2)$. 

Another important and still unresolved problem in the representation theory of $G$ is to determine whether a projective module for a Frobenius kernel of $G$ has a structure as a $G$-module.  The expectation that this should be true dates back to work of Humphreys and Verma \cite{HV}.  Donkin later conjectured that such structures should arise from tilting modules for $G$ \cite{Don}, which is now known to hold for $p \geq 2(h-2)$ (and some small rank cases). In some sense, this conjecture was 
implicit in \cite{HV}, which predates the concept of tilting modules by many years. The general method involved tensoring by the Steinberg representation and locating an important $G$-summand that 
is known to be a tilting module by a result due to Pillen \cite{P}. 

In \cite{BNPS20}, the authors found counterexamples to (i) Jantzen's Question (JQ) and (ii) the Tilting Module Conjecture (TMC).  For exact statements of JQ and the TMC, see 
Question~\ref{Q:JQ} and Conjecture~\ref{C:TMC}. The counterexamples occur for the simple algebraic group of type $\text{G}_2$ in characteristic $2$, where it was shown that the tilting module $T(2,2)$ is not indecomposable over the Frobenius kernel of $G$ and that $\nabla(2,1)$ does not have a good $p$-filtration.  As the first surprising example of its kind, it provided some clarity regarding the TMC, explaining if nothing else why it had resisted proof for nearly 30 years. After that paper was written, a number of questions still remained to be answered. For instance, was this example an anomaly, or were there many others like it?  There is also the basic question as to what features about the representation theory and cohomology give rise to this counterexample and what causes the tilting module $T(2,2)$ to split over the Frobenius kernel. It should be noted that Kildetoft and Nakano, and Sobaje made explicit connections between the TMC and good $p$-filtrations on $G$-modules (cf. \cite{KN}, \cite{So}).  

\subsection{}\label{S:main} The main goal of this paper is to show that the type $\rm{G}_2$ examples were in fact not anomalies.  We show how to produce large families of counterexamples to JQ and the TMC. Our results are summarized in the following theorem. 

\begin{theorem}\label{T:main} Let $G$ be a simple algebraic group over an algebraically closed field of characteristic $p>0$. Assume that the underlying root system $\Phi$ is not of type 
$\rm{A}_{n}$ or $\rm{B}_{2}$. Then there exists a prime $p$ for which $G$ produces a counterexample to Jantzen's Question and the Tilting Module Conjecture. 
\end{theorem} 

In particular, we can exhibit counterexamples to JQ and the TMC for the following groups. 
\begin{itemize} 
\item $\Phi=\rm{B}_{n}$, $p=2$, $n\geq 3$;
\item $\Phi=\rm{C}_{n}$, $p=3$, $n\geq 3$;
\item $\Phi=\rm{D}_{n}$, $p=2$, $n\geq 4$;
\item $\Phi=\rm{E}_{n}$, $p=2$, $n=6,7,8$;
\item $\Phi=\rm{F}_{4}$, $p=2,3$; 
\item $\Phi=\rm{G}_{2}$, $p=2$. 
\end{itemize} 

The counterexamples in type $\rm{C}_n$ at $p = 3$ are particularly interesting as 3 is a {\em good} prime for the root system.  We anticipate that counterexamples exist for all $\rm{C}_n$ when $n = p$ is prime.  In general, the question of when the TMC holds still seeems quite subtle.  
In \cite{BNPS21} and \cite{And19}, it was shown that JQ and the TMC holds when $\Phi=\rm{B}_{2}$ for all $p$. Moreover, it was shown
in \cite{BNPS21} that the TMC holds in type $\rm{G}_2$ for all $p > 2$ and for all primes for $\Phi=\rm{A}_{n}$ when $n\leq 3$. 
It is still possible that the TMC for $\Phi=\rm{A}_{n}$ holds for all $n$ and all $p$. 

The paper is organized as follows. In Section 2, we introduce the notation and conventions for the paper. Two reduction theorems (cf. Theorems~\ref{T:JQLevi} and ~\ref{T:TMCLevi}) are established for JQ and the TMC that enable one to use counterexamples in low rank to establish counterexamples in large rank via Levi subgroups.

Section 3 focuses on new methods to use $\text{Ext}^{1}$ structural information to produce counterexamples to the TMC. 
Our new examples also give some indication as to why the failure of the TMC is occurring.  Namely, in each case one finds that there are (restricted) dominant weights $\lambda$ and $\mu$ such that the $G/G_{1}$-module
$\Ext_{G_1}^1(L(\lambda),L(\mu))$ is not a tilting module.  Indeed, the counterexamples in this paper were actually discovered by first looking for such behavior in $\text{Ext}$-groups, having already observed its effect on the 
$\Phi=\rm{G}_{2}$, $p=2$, example in \cite{BNPS20}.  
 
In Section 4, counterexamples to JQ and the TMC are worked out for $\Phi=\rm{B}_{3}$, $\rm{C}_{3}$, and $\rm{D}_{4}$ via the methods in Section 3. In the following section (Section 5), the main theorem is proved. 
In Section 6, we show how to produce other counterexamples using the Jantzen filtration. In particular, we exhibit other counterexamples to the  TMC in characteristic $2$ for all groups of type $\text{B}_n$ with $n\geq 3$.  We also construct additional counterexamples in type $\text{C}_n$ for $p=3$ and $n\geq 3$.   

At the end of the paper, in Section 7, we present some open problems involving the connections with the TMC and the question of whether $\Ext_{G_1}^1(L(\lambda),L(\mu))$ is a tilting $G/G_{1}$-module. We view these 
important observations as  opening the door for future investigation. 

\subsection{Acknowledgements} We thank Henning Andersen for comments and suggestions on an earlier version of our manuscript.

\section{Preliminaries} 

\subsection{Notation}\label{S:notation} In this paper we will generally follow the standard conventions in \cite{rags}. Throughout this paper $k$ is an algebraically closed field of characteristic $p>0$. 
Let $G$ be a connected semisimple algebraic group scheme 
defined over ${\mathbb F}_{p}$. The Frobenius morphism is denoted by $F$ and the $r$th Frobenius kernel will be denoted by $G_{r}$.  Given a maximal torus $T$, 
the root system $\Phi$ is associated to the pair $(G,T)$. Let $\Phi^+$ be a set of positive roots and
$\Phi^{-}$ be the corresponding set of negative roots. The set of simple roots determined by $\Phi^+$ is $\Delta=\{\alpha_1,\dots,\alpha_{l}\}$. The ordering of simple roots is 
described in \cite{Hum1} following Bourbaki. 

Let $B$ be the Borel subgroup given by the set of negative roots and let $U$ be
the unipotent radical of $B$. More generally, if $J\subseteq \Delta$, let $P_{J}$ be the
parabolic subgroup relative to $-J$, $L_{J}$ be the Levi factor of $P_{J}$ and $U_{J}$ be the
unipotent radical. Let $\Phi_{J}$ be the root subsystem in $\Phi$ generated by the simple roots in $J$, 
with positive subset $\Phi_{J}^{+} = \Phi_{J}\cap\Phi^{+}$.

Let $\mathbb E$ be the Euclidean space associated with $\Phi$, and
denote the inner product on ${\mathbb E} $ by $\langle\ , \
\rangle$. Set $\alpha_{0}$ to be the highest short root.
Let $\rho$ be the half sum of positive roots and $\alpha^{\vee}$ be the coroot corresponding to
$\alpha\in \Phi$.   The Coxeter number associated to $\Phi$ is $h=\langle \rho,\alpha_{0}^{\vee} \rangle +1$. The Weyl group 
associated to $\Phi$ will be donated by $W$, and, for any $J\subseteq \Delta$, let $W_{J}$ be the subgroup of $W$ generated by reflections corresponding to simple 
roots in $J$. Let $w_0$ (resp. $w_{J,0}$) denote the longest word of $W$ (resp. $W_{J}$, for $J \subseteq \Delta.$). Set $\rho_J$ to be the half-sum of all the roots in $\Phi_J^+$.

Let $X:=X(T)$ be the integral weight lattice spanned by the fundamental weights $\{\omega_1,\dots,\omega_l\}$, $X^{+}$ be the dominant weights for $G$, and $X_{r}$ be the $p^{r}$-restricted weights. 
Moreover, for $J\subseteq \Delta$, let $X_{J}^{+}$ be the weights that are dominant on $J$. That is, $X_{J}^+ := \{\la \in X ~|~ \langle \la,\a^{\vee}\rangle \geq 0 ~\forall~ \a \in J\}$.

Let $\tau:G \rightarrow G$ be the Chevalley antiautomorphism of $G$ that is the identity morphism when restricted to $T$ (see \cite[II.1.16]{rags}).  Given a finite dimensional $G$-module $M$ over $k$, the module $^{\tau}M$ is $M^*$ (the ordinary $k$-linear dual of $M$) as a $k$-vector space, with action $g.f(m)=f(\tau(g).m)$.  This defines a functor from $G$-mod to $G$-mod that preserves the character of $M$.  In particular, it is the identity functor on all simple and tilting modules. 

The following result will be used throughout this section. Note that we follow the convention in \cite{rags} that the set $\mathbb{N}$ includes $0$. 

\begin{lemma}\label{L:wmu} Let $\mu \in X^+$ and $w \in W$ be such that $\mu - w \mu \in \mathbb{N}J$.
Then there exists $w_J \in W_J$ with $w\mu=w_J\mu.$ 
\end{lemma}

\begin{proof} There exists $u \in W_J$ such that $\langle u w \mu, \beta^{\vee} \rangle \geq 0$ for all $\beta \in J,$ (i.e., $u w \mu $ is  $J$-dominant).
Since $u \in W_J$, $w\mu - u w \mu \in \mathbb{Z}J.$  With this and the hypothesis,
$\mu -  u w \mu= (\mu - w \mu) + (w\mu - u w \mu) \in \mathbb{Z}J.$  However, $\mu \geq u w \mu$, thus $\mu -  u w\mu \in \mathbb{N}J$. 

If $u w \mu \in X^+$, then $u w \mu = \mu$ and one may choose $w_J = u^{-1} \in W_J$. Thus, $w \mu = w_J \mu.$ 
If $u w \mu$ is not in $X^{+}$, then there exists a $\beta \in \Delta\setminus J$ such that $\langle u w \mu, \beta^{\vee} \rangle < 0,$ which implies that $\langle \mu - u w \mu, \beta^{\vee} \rangle > 0.$ This contradicts the fact that $\mu -  u w\mu \in \mathbb{N}J.$
\end{proof}

\subsection{Representations}\label{S:reps}
 For $\lambda\in X^{+}$, there are four fundamental families of $G$-modules (each having highest weight $\lambda$): $L(\lambda)$ (simple), $\nabla(\lambda)=H^{0}(\lambda)$ (costandard/induced), $\Delta(\lambda) =V(\lambda)$ (standard/Weyl), and $T(\lambda)$ (indecomposable tilting).
Let $\text{St}_r = L((p^r-1)\rho)$ be the $r$th Steinberg module. For $\lambda\in X_{r}$, let $Q_{r}(\lambda)$ denote the $G_{r}$-projective cover (equivalently, injective hull) of $L(\lambda)$ as a $G_{r}$-module. For $\lambda\in X$, if $\widehat{L}_{r}(\lambda)$ is the corresponding simple $G_{r}T$-module, let $\widehat{Q}_{r}(\lambda)$ denote the $G_{r}T$-projective cover (equivalently, injective hull) of $\widehat{L}_{r}(\lambda)$.  
For $\lambda\in X^+$, write $\lambda = \lambda_0 + p^r\lambda_1$ with $\lambda_0\in X_r$ 
and $\lambda_1\in X^+$. Define $\nabla^{(p,r)}(\lambda) = L(\lambda_0)\otimes \nabla(\lambda_1)^{(r)}$, where $(r)$ denotes the twisting of the module action by the 
$r$th Frobenius morphism. 

A $G$-module $M$ has a {\em good filtration} 
(resp. {\em good $(p,r)$-filtration}) if and only if $M$ has a filtration with factors of the form $\nabla(\mu)$ (resp. $\nabla^{(p,r)}(\mu)$) for suitable $\mu\in X^+$.  
In the case when $r=1$, good $(p,1)$-filtrations are often referred to as good $p$-filtrations. 

For each $\lambda \in X_J^+$ there is a simple $L_J$-module $L_J(\lambda)$, a standard/Weyl module $\Delta_J(\lambda)$, a costandard/induced module $\nabla_J(\lambda),$ and an indecomposable tilting module $T_J(\la)$.  Specifically, for $\la \in X_{J}^+,$
$\nabla_J(\la)= \ind_{B}^{P_J}(\la)\cong \ind_{B\cap L_J}^{L_J}(\la)$. Furthermore, for $\la\in X$ and $r \geq 1$, set 
$$\hZ'_{J,r}(\la) = \ind_{B_rT}^{(P_J)_rT}(\la)  \cong \ind_{(B_r \cap L_J)T}^{(L_J)_rT}(\la)$$ \
and $\hQ_{J,r}(\la)$ to be the $(L_J)_rT$-injective hull of the simple $(L_J)_rT$-module $\hL_{J,r}(\la)$. 
When $J = \Delta$, we simply denote $\hZ'_{J,r}(\la)$ by $\hZ'_r(\la)$.  

Let $r \ge 1$ and write $\lambda = \lambda_0 + p^r\lambda_1$, where $\lambda_0 \in X_r(T)$ and $\lambda_1 \in X_J^+$.  By $\nabla_J^{(p,r)}(\lambda)$ denote the module
$$\nabla_J^{(p,r)}(\lambda) = L_J(\lambda_0) \otimes \nabla_J(\lambda_1)^{(r)}.$$
An $L_J$-module with a filtration whose factors have this form is said to have a good $(p,r)$-filtration.
In \cite[Proposition II.2.11 and II.5.21]{rags}, it is observed that one has the following facts with  $\la, \mu \in X^+$: 
\begin{itemize}
\item[(i)] $\nabla_J(\la)= \bigoplus_{ \nu \in \mathbb{N}J} \nabla(\la)_{\la-\nu}$,
\item[(ii)] $L_J(\la)= \bigoplus_{ \nu \in \mathbb{N}J} L(\la)_{\la-\nu}$,
\item[(iii)] $[\nabla(\la) : L(\mu)] = [\nabla_J(\la): L_J(\mu)],$ whenever $\la- \mu \in \mathbb{N}J.$
\end{itemize} 
Furthermore, the arguments in \cite{Don83} can  be adapted to yield for $\la, \mu, \la', \mu' \in X$:
\begin{itemize}
\item[(iv)] $\hZ'_{J,r}(\la) = \bigoplus_{ \nu \in \mathbb{N}J} \hZ_r'(\la)_{\la-\nu}$,
\item[(v)] [$\hZ'_r(\la) : \hL_r(\mu)] = [\hZ'_{J,r}(\la): \hL_{J,r}(\mu)],$ whenever $\la- \mu \in \mathbb{N}J$,
\item[(vi)] If $\la -\mu  \in \mathbb{N}J$ and $\langle\la - \la', \beta^{\vee} \rangle = \langle\mu - \mu', \beta^{\vee} \rangle = 0$ for all $\beta \in J,$ then $$[\hZ'_{J,r}(\la): \hL_{J,r}(\mu)] =[\hZ'_{J,r}(\la'): \hL_{J,r}(\mu')].$$
\end{itemize}

The following proposition discusses the restriction of tensor products of $G$-modules with unique highest weights to  Levi subgroups, with an immediate application involving induced modules.

\begin{prop}\label{P:tensor}
Let $\lambda, \mu \in X^+$. 
\begin{itemize} 
\item[(a)]  If $Y(\la)$ and $Y(\mu)$ are $G$-modules with unique highest weights $\la$ and $\mu$, respectively,  then we have an equality of $L_J$-modules
\begin{align*}
 \left(\bigoplus_{\nu \in \mathbb{N}J} (Y(\la) )_{\lambda-\nu}\right) \otimes  \left(\bigoplus_{\nu \in \mathbb{N}J} (Y(\mu) )_{\mu-\nu}\right)
&= \bigoplus_{\nu \in \mathbb{N}J} \left(Y(\la) \otimes Y(\mu)\right)_{\lambda+ \mu-\nu}.
\end{align*}
\item[(b)] There exists an equality of $L_J$-modules
\begin{align*}
\nabla_J^{(p,r)}(\lambda) & = \bigoplus_{\nu \in \mathbb{N}J} \nabla^{(p,r)}(\lambda)_{\lambda-\nu}.
\end{align*}
\end{itemize} 
\end{prop}

\begin{proof} (a) Clearly we have that
\begin{align*}
 \left(\bigoplus_{\nu \in \mathbb{N}J} Y(\lambda)_{\lambda -\nu}\right) \otimes \left(\bigoplus_{\nu \in \mathbb{N}J} Y(\mu)_{\mu -\nu}\right)
& \subseteq \bigoplus_{\nu \in \mathbb{N}J} \left(Y(\lambda) \otimes Y(\mu)\right)_{\lambda+\mu-\nu}.
\end{align*}
On the other hand, if $Y(\lambda)_{\gamma}$ is a nonzero weight space and $\gamma$ is not of the form $\lambda - \nu$ for some $\nu \in \mathbb{N}J$, then there is some simple root $\alpha \in \Delta\backslash J$ such that when $\lambda-\gamma$ is expressed in the basis of simple roots, the coefficient $n_{\alpha}$ of $\alpha$ has $n_{\alpha}>0$.  For there to be a weight $\sigma$ of $Y(\mu)$ such that $\gamma+\sigma=\lambda+\mu-\nu$ with $\nu \in \mathbb{N}J$, we need that
$$\lambda+\mu - (\gamma+\sigma) \in \mathbb{N}J$$
or equivalently that
$$(\lambda - \gamma) +(\mu - \sigma) \in \mathbb{N}J.$$
Thus we have that in writing $\mu - \sigma$ in the basis $\Delta$, the coefficient of $\alpha$ is $-n_{\alpha}$, and $-n_{\alpha}<0$.  Since $\mu \ge \sigma$, this cannot happen.  A similar argument, reversing the roles of $\lambda$ and $\mu$, shows then that we have an equality
\begin{align*}
 \left(\bigoplus_{\nu \in \mathbb{N}J} (Y(\la) )_{\lambda-\nu}\right) \otimes  \left(\bigoplus_{\nu \in \mathbb{N}J} (Y(\mu) )_{\mu-\nu}\right)
&= \bigoplus_{\nu \in \mathbb{N}J} \left(Y(\la) \otimes Y(\mu)\right)_{\lambda+ \mu-\nu}.
\end{align*}

(b) With (i) and (ii), applying part (a) to $Y(\lambda_0)=L(\lambda_0),$ with $\lambda_0 \in X_r,$ and $Y(p^r\lambda_1)= \nabla(\la_1)^{(r)},$ with $ \la_1 \in X^+,$ immediately yields the result. 

\end{proof}

\subsection{Jantzen's Question and the Tilting Module Conjecture} For a complete description of Jantzen's Question, Donkin's $(p,r)$-Filtration Conjecture, and 
Donkin's Tilting Module Conjecture, we refer the reader to \cite[Section 2.2]{BNPS20}. In 1980, Jantzen \cite{Jan80} asked the following question which lead to the formulation of the $(p,r)$-Filtration 
Conjecture. 

\begin{quest}\label{Q:JQ} For $\lambda\in X^{+}$, does $\nabla(\lambda)$ admit a good $(p,r)$-filtration? 
\end{quest} 

The Tilting Module Conjecture, introduced by Donkin in 1990 \cite{Don}, states that 
a projective indecomposable module for $G_r$ can be realized as an indecomposable tilting $G$-module. 
 
\begin{conj}\label{C:TMC} For all $\lambda\in X_{r}$, 
$$T((p^{r}-1)\rho+\lambda)|_{G_{r}T}=\widehat{Q}_{r}((p^{r}-1)\rho+w_{0}\lambda).$$
\end{conj}
\noindent 
An alternative and equivalent  formulation of the conjecture is that 
$$T(2(p^{r}-1)\rho+w_{0}\lambda)|_{G_{r}T}=\widehat{Q}_{r}(\lambda)$$
for all $\lambda\in X_{r}$. 

One direction of Donkin's $(p,r)$-Filtration Conjecture is equivalent to the following.

\begin{conj}\label{C:stla} For $\la \in X_r$, $\St_r\otimes L(\la)$ is a tilting module.
\end{conj}

In \cite{So}, it was shown that an affirmative answer to Question~\ref{Q:JQ} and Conjecture~\ref{C:stla} implies that Conjecture~\ref{C:TMC} holds.

\subsection{Jantzen's Question and Levi Subgroups}

The following result demonstrates that JQ is compatible with restricting to Levi subgroups. This theorem will be used later to produce counterexamples to JQ by analyzing examples in low rank. 

\begin{theorem}\label{T:JQLevi}
Let $L_J$ be a Levi subgroup of $G$, and let $\lambda \in X^+$.  If $\nabla(\lambda)$ has a good $(p,r)$-filtration as a $G$-module, then $\nabla_J(\lambda)$ has a good $(p,r)$-filtration as an $L_J$-module.
Moreover, one obtains that
$[\nabla(\la) : \nabla^{(p,r)}(\mu)] = [\nabla_J(\la): \nabla_J^{(p,r)}(\mu)],$
whenever $\la- \mu \in \mathbb{N}J.$
\end{theorem}

\begin{proof}
As an $L_J$-module, $\nabla_J(\lambda)$ is a direct summand of $\nabla(\lambda)$ and, from \ref{S:reps}(i), is given explicitly as the sum of weight spaces
$$\bigoplus_{\nu \in \mathbb{N}J} \nabla(\lambda)_{\lambda-\nu}.$$  In particular, there is a projection map $\pi:\nabla(\lambda) \rightarrow \nabla_J(\lambda)$, which is a homomorphism of $L_J$-modules.  Additionally, it follows that on each $T$-weight space, $\pi$ is the zero map if the weight is not of the form $\lambda-\nu$ for $\nu \in \mathbb{N}J$, otherwise $\pi$ is an isomorphism onto the corresponding weight space in the image.

If $\nabla(\lambda)$ has a good $(p,r)$-filtration, then there is a filtration of $G$-submodules
$$\{0\}=F_0 \subseteq F_1 \subseteq \cdots \subseteq F_m = \nabla(\lambda)$$
and dominant weights $\mu_1,\ldots,\mu_m$ such that each $F_i/F_{i-1} \cong \nabla^{(p,r)}(\mu_i)$.  We can apply $\pi$ to this filtration, obtaining a filtration on $\nabla_J(\lambda)$:
$$\{0\}=\pi(F_0) \subseteq \pi(F_1) \subseteq \cdots \subseteq \pi(F_m) = \nabla_J(\lambda).$$
For each $i$, we have (by restricting $\pi$ and composing with a quotient map) an $L_J$-module homomorphism
$$\pi_i: F_i \rightarrow \pi(F_i)/\pi(F_{i-1}),$$
and clearly $F_{i-1}$ is contained in the kernel of $\pi_i$.  Thus we obtain a homomorphism of $L_J$-modules
$$\nabla^{(p,r)}(\mu_i) \rightarrow \pi(F_i)/\pi(F_{i-1}).$$
As noted above, this homomorphism is nonzero if and only if $\nabla^{(p,r)}(\mu_i)$ has nonzero weight vectors of the form $\lambda-\nu$ for $\nu \in \mathbb{N}J$.  Since $\lambda \ge \mu_i$, it follows (by an argument as the proof of Proposition \ref{P:tensor}) that it is necessary (and clearly sufficient) that $\la - \mu_i \in \mathbb{N}J$.

Thus we have that
$$\pi(F_i)/\pi(F_{i-1}) \cong \nabla_J^{(p,r)}(\mu_i)$$
if $\lambda-\mu_i \in \mathbb{N}J$, and
$$\pi(F_i)/\pi(F_{i-1}) = \{0\}$$
otherwise, so that $\nabla_J(\lambda)$ has a good $(p,r)$-filtration.
\end{proof}

\subsection{Injective Hulls and Levi Subgroups} 
 It was first observed by Donkin \cite{Don} that the restriction of a projective indecomposable $G_{r}T$ module to $(L_{J})_{r}T$ yields  a  decomposition  that is similar to those listed in  Section \ref{S:reps}.  As a proof was only outlined in \cite{Don}, for the reader's convenience, 
we include a self-contained proof that will also be used in \cite{BNPS23}.

\begin{theorem}\label{T:QLevi}  \cite[Proposition 2.7]{Don} Let $L_J$ be a Levi subgroup of G and  $\la \in X_r.$ Then we have an equality of $(L_J)_rT$-modules.
$$\hQ_{J,r}((p^r-1)\rho + w_{J,0}\la) = \bigoplus_{\nu \in \mathbb{N}J} \widehat{Q}_r((p^r-1)\rho+w_0 \la)_{(p^r-1)\rho+\la-\nu}.$$
\end{theorem}

\begin{proof} To prove the theorem, we define the $(L_J)_rT$-summand $M$ of $\widehat{Q}_r((p^r-1)\rho+w_0 \la)$ via 
$$M=\bigoplus_{\nu \in \mathbb{N}J} \widehat{Q}_r((p^r-1)\rho+w_0 \la)_{(p^r-1)\rho+\la-\nu}.$$
The restriction of $\widehat{Q}_r((p^r-1)\rho+w_0 \la)$ to $(L_J)_rT$ is still injective and projective. Therefore, $M$ is also injective and projective as an $(L_J)_rT$-module. 
It suffices to show that $$\text{ch } M = \text{ch } \hQ_{J,r}((p^r-1)\rho + w_{J,0}\la).$$  Verification of this latter claim will be carried out over the next several steps. 
\vskip .25cm 
(1) {\em Formal Character of $\hQ_{J,r}((p^r-1)\rho +w_{J,0}\la)$:}\label{S:charQJr}
The $(L_J)_rT$-module 
$$\nabla_J((p^r-1)\rho)= L_J((p^r-1) \rho)=\hZ'_{J,r}((p^r-1) \rho)$$ 
is irreducible, injective and projective. In addition, 
\begin{align*}
&\Hom_{(L_J)_rT}(L_J((p^r-1)\rho+w_{J,0}\la) ,\hZ'_{J,r}((p^r-1) \rho)\otimes L_J(\la))\\
&= \Hom_{(L_J)_rT}(L_J((p^r-1)\rho+w_{J,0}\la) \otimes L_J(-w_{J,0}\la), L_J((p^r-1) \rho)) \\
&= k.
\end{align*}
Therefore, $\hQ_{J,r}((p^r-1)\rho +w_{J,0}\la)$ is an $(L_J)_rT$-summand of $\hZ'_{J,r}((p^r-1) \rho)\otimes L_J(\la).$ The latter has a $\hZ'_{J,r}$-filtration with factors of the form $\hZ'_{J,r}((p^r-1) \rho+\ga),$ with $\ga$ being a weight of $L_J(\la)$. The weights of $L_J(\la)$ are all of the form $w\mu$ with $\mu \in X_J^+,$ $w \in W_J,$ and $\la-\mu \in \mathbb{N}J.$ Note that $\la \in X_r,$ $\mu \in X_J^+,$  and $\la-\mu \in \mathbb{N}J$ implies that $\mu \in X^+.$
Hence, $\hQ_{J,r}((p^r-1)\rho +w_{J,0}\la)$ has an $(L_J)_rT$-filtration with factors of the form $\hZ_{J,r}'((p^r-1)\rho + w \mu)$ with  $\mu \in   \{\ga \in X^+ | \la - \ga \in \mathbb{N}J\}$ and $w \in W_J.$ 
By making use of Brauer-Humphreys reciprocity \cite[II.11.4]{rags} and $W_J$-invariance \cite[II.9.16(5)]{rags},  one obtains 
\begin{eqnarray*}
&\text{ch}& \hQ_{J,r}((p^r-1)\rho +w_{J,0}\la)\\
&=&  \sum_{\{\mu \in X^+ | \la - \mu \in \mathbb{N}J\}} \frac{1}{|\text{Stab}_{W_J}(\mu)|} \sum_{w \in W_J} [\hQ_{J,r}((p^r-1)\rho + w_{J,0}\la): \hZ_{J,r}'((p^r-1)\rho +w\mu)] \\
&&\cdot \; \text{ch } \hZ_{J,r}'((p^r-1)\rho +w\mu)\\
&=&  \sum_{\{\mu \in X^+ | \la - \mu \in \mathbb{N}J\}}\frac{1}{|\text{Stab}_{W_J}(\mu)|} \sum_{w \in W_J} [ \hZ_{J,r}'((p^r-1)\rho +w\mu): L_J((p^r-1)\rho + w_{J,0}\la)]\\
&& \cdot \; \text{ch } \hZ_{J,r}'((p^r-1)\rho +w\mu)\\
&=&  \sum_{\{\mu \in X^+ | \la - \mu \in \mathbb{N}J\}}  \frac{1}{|\text{Stab}_{W_J}(\mu)|} \cdot [ \hZ_{J,r}'((p^r-1)\rho +\mu): L_J((p^r-1)\rho + w_{J,0}\la)]\\
&& \cdot \sum_{w \in W_J}\text{ch } \hZ_{J,r}'((p^r-1)\rho +w\mu).\\
\end{eqnarray*}

\noindent 
Using $(L_J)_rT$-duality \cite[II.9.2]{rags}, we obtain
$\displaystyle{ [ \hZ_{J,r}'((p^r-1)\rho +\mu): L_J((p^r-1)\rho + w_{J,0}\la)]}$
 \begin{eqnarray*}
 &=&  [ \hZ_{J,r}'(2(p^r-1)\rho_J-(p^r-1)\rho -\mu): L_J(-w_{J,0}(p^r-1)\rho -\la)]\\
 &=& [ \hZ_{J,r}'(-w_{J,0}(p^r-1)\rho -\mu): L_J(-w_{J,0}(p^r-1)\rho -\la)].
 \end{eqnarray*}
 
 \noindent 
 Hence,
 \begin{eqnarray*}
&\text{ch}& \hQ_{J,r}((p^r-1)\rho +w_{J,0}\la)\\
&= & \sum_{\{\mu \in X^+ | \la - \mu \in \mathbb{N}J\}}  \frac{1}{|\text{Stab}_{W_J}(\mu)|} \cdot [ \hZ_{J,r}'(-w_{J,0}(p^r-1)\rho -\mu): L_J(-w_{J,0}(p^r-1)\rho -\la)]\\
&& \cdot \sum_{w \in W_J}\text{ch } \hZ_{J,r}'((p^r-1)\rho +w\mu).\\
\end{eqnarray*}
Observe also that the highest weight of $\hQ_{J,r}((p^r-1)\rho +w_{J,0}\la)$ is $(p^r-1)\rho + \la.$
\vskip .25cm 
(2) {\em Formal character of $\widehat{Q}_r((p^r-1)\rho+w_0 \la)$:}\label{S:charQr}
If one applies the calculation of step (1) to the case $J=\Delta,$ one obtains a filtration of $\widehat{Q}_r((p^r-1)\rho+w_0 \la)$ with factors of the form $\hZ_r'((p^r-1)\rho +w\mu)$ where $\mu \in X^+$ and $w \in W.$ The equivalent of the last equation in step (1) is 

\begin{eqnarray*}
\text{ch } \widehat{Q}_r((p^r-1)\rho+w_0 \la)
&=&  \sum_{\{\mu \in X^+| \mu \leq \la\} } \frac{1}{|\text{Stab}_{W}(\mu)|}
 \cdot [ \hZ_r'((p^r-1)\rho -\mu): L((p^r-1)\rho -\la)]\\
  &&\cdot \sum_{w \in W}\text{ch } \hZ_r'((p^r-1)\rho +w\mu).
\end{eqnarray*}
\vskip .25cm 
(3) {\em Formal Character of M:}\label{S:charM}
The aforementioned $\hZ_r'$-filtration of $\widehat{Q}_r((p^r-1)\rho+w_0 \la)$ induces a filtration on the $(L_J)_rT$-summand $M$.  
Analogous to the argument in the proof of Theorem \ref{T:JQLevi}, the resulting filtration will have factors of the form $\hZ'_{J,r}(\s)$.
One observes that  a factor of the form $\hZ_r'((p^r-1)\rho + w \mu)$ can only contribute to the filtration of $M$ if $$(p^r-1)\rho +\la - (p^r-1)\rho -w \mu= (\la - \mu)+(\mu - w\mu) \in \mathbb{N}J.$$
Since both $\la - \mu \geq 0$ and $\mu - w \mu \geq 0$, it follows that $\la - \mu \in \mathbb{N}J$ and that $\mu - w \mu \in \mathbb{N}J.$  From Lemma \ref{L:wmu}, the latter implies that $w\mu = w_J\mu$ for some $w_J \in W_J$. If these conditions are satisfied, then it follows from \ref{S:reps}(iv) that the  resulting $(L_J)_rT$-factor of $M$ is  isomorphic to $\bigoplus_{ \nu \in \mathbb{N}J} \hZ'_{r}((p^r-1)\rho+w_J\mu)_{(p^r-1)\rho+\la - \nu}
=\hZ'_{J,r}((p^r-1)\rho+w_J\mu).$ 
One obtains from the above discussion and step (2) that 
\begin{eqnarray*}
\text{ch } M
 &=& \sum_{\{\mu \in X^+ | \la - \mu \in \mathbb{N}J\}} \frac{1}{|\text{Stab}_{W_J}(\mu)|} \cdot [\hZ_r'((p^r-1)\rho -\mu):L((p^r-1)\rho -\la)] \\
 &&\cdot  \sum_{w \in W_J} \text{ch } \hZ_{J,r}'((p^r-1)\rho +w\mu).
\end{eqnarray*}
Since $(p^r-1)\rho -\mu-((p^r-1)\rho -\la) =\la - \mu \in \mathbb{N}J$, 
by \ref{S:reps}(v), it follows that
$$ [\hZ_r'((p^r-1)\rho -\mu):L((p^r-1)\rho -\la)] =[\hZ_{J,r}'((p^r-1)\rho -\mu):L_{J}((p^r-1)\rho -\la)]$$
and
\begin{eqnarray*}
\text{ch } M
 &=& \sum_{\{\mu \in X^+ | \la - \mu \in \mathbb{N}J\}} \frac{1}{|\text{Stab}_{W_J}(\mu)|} \cdot [\hZ_{J,r}'((p^r-1)\rho -\mu):L_J((p^r-1)\rho -\la)] \\
 &&\cdot  \sum_{w \in W_J} \text{ch } \hZ_{J,r}'((p^r-1)\rho +w\mu).\\
\end{eqnarray*}
\vskip .25cm 
(4) {\em Comparison and Completion of Proof:}\label{S:completion}
We claim that for $\la \in X_r$ and $\mu \in X^+$ with $\la - \mu \in \mathbb{N}J$  $$[\hZ_{J,r}'((p^r-1)\rho -\mu):L_J((p^r-1)\rho -\la)]= [\hZ_{J,r}'(-w_{J,0}(p^r-1)\rho -\mu): L_J(-w_{J,0}(p^r-1)\rho -\la)].$$ Note the pairs of weights on each side of the equation both differ by $\la - \mu.$ 
In addition, $\langle (p^r-1)\rho+ w_{J,0}(p^r-1)\rho, \beta^{\vee} \rangle =0,$ for all $\beta \in J.$ The claim follows now from \ref{S:reps}(vi).
 Therefore, comparing the final equations of steps (1) and (3) yields the assertion of Theorem \ref{T:QLevi}.

\end{proof} 

\subsection{The Tilting Module Conjecture and Levi Subgroups} Using Theorem~\ref{T:QLevi},
we can show that the validity of the TMC is compatible with restriction to 
Levi subgroups. 

\begin{theorem}\label{T:TMCLevi} Let $L_J$ be a Levi subgroup of G. If the Tilting Module Conjecture holds for $G$, then it also holds for $L_J.$ 
\end{theorem}

\begin{proof} 
 Verifying the TMC for $L_J$ is equivalent to showing the following: given a weight $\la$ that is $p^r$-restricted on $L_J,$ the unique indecomposable $L_J$-summand  containing the  highest weight $(p^r-1)\rho_J + \la$ in the tensor product $L_J((p^r-1)\rho_J) \otimes L_J(\la)$  remains indecomposable as an $(L_J)_rT$-module.

Note that  $\langle (p^r-1)\rho -(p^r-1)\rho_J , \beta^{\vee} \rangle =0$ for all  $\beta \in J$ and 
 that, given any simple $p^r$-restricted $L_J$-module $V$,  one can always find $\la \in X_r$ such that the highest weights of $V$ and $L_J(\la)$ agree on all weight components corresponding to $J.$ 
 
 It is therefore sufficient to show that, for all $\la \in X_r,$
 the unique indecomposable $L_J$-summand  containing the  highest weight $(p^r-1)\rho + \la$ in the tensor product $L_J((p^r-1)\rho) \otimes L_J(\la)$  is indecomposable as an $(L_J)_rT$-module. We will denote this $L_J$-summand, which is tilting, by 
 $T_J((p^r-1)\rho + \la).$
 
 The indecomposable $G$-tilting module $T((p^r-1)\rho + \la)$ appears as the unique $G$-summand containing the weight $(p^r-1)\rho + \la$ in $L((p^r-1) \rho) \otimes L(\la).$ 
 From \ref{S:reps}(ii) and Lemma \ref{P:tensor},
 \begin{eqnarray*}
 L_J((p^r-1)\rho) \otimes L_J(\la)
 &=& \bigoplus_{\nu \in \mathbb{N}J}\left(L((p^r-1) \rho) \otimes L(\la)\right)_{(p^r-1)\rho+\la-\nu}.
 \end{eqnarray*}
  It follows that $T_J((p^r-1)\rho+ \la)$ appears as an $L_J$-summand of  $T((p^r-1)\rho + \la)$. More precisely, it is a summand of the $L_J$-module 
$N=\bigoplus_{\nu \in \mathbb{N}J} T((p^r-1)\rho + \la)_{(p^r-1)\rho+\la-\nu}.$

The TMC for $G$ implies that $T((p^r-1)\rho + \la)=\hQ_r((p^r-1)\rho+w_0 \la)$, as $G_rT$-modules.  We obtain from Theorem \ref{T:QLevi} that 
$$N= \bigoplus_{\nu \in \mathbb{N}J}\hQ_r((p^r-1)\rho+w_0 \la)_{(p^r-1)\rho+\la-\nu}= \hQ_{J,r}((p^r-1)\rho + w_{J,0}\la),$$
as $(L_J)_rT$-modules. Hence, $N$ and $T_J((p^r-1)\rho+ \la)$ are indecomposable as $(L_J)_rT$-modules.
\end{proof}

\begin{remark} 
Donkin observed in \cite[Proposition 1.5 (ii)]{Don} that indecomposable tilting modules behave nicely when restricted to Levi subgroups. More precisely, he showed that for any $\la \in X^+$
\begin{equation}\label{E:Donkin}
T_J( \la)=\bigoplus_{\nu \in \mathbb{N}J} T( \la)_{\la-\nu}.
\end{equation} 
Now Equation \ref{E:Donkin} together with Theorem \ref{T:QLevi} may also be used to  prove Theorem \ref{T:TMCLevi}.
In particular, in the preceding proof of Theorem \ref{T:TMCLevi}, it would immediately follow that $T_J((p^r-1)\rho+ \la)$ was in fact equal to the module $N$, rather than simply being a summand; a conclusion that is also reached at the end of the proof. 
 
\end{remark}

In subsequent sections, the contrapositives of Theorems \ref{T:JQLevi} and \ref{T:TMCLevi} will be used to obtain infinite families of counterexamples to JQ and the TMC from low rank counterexamples.  To conclude this section, we record the relationship between the validity of the TMC and of the character of $\nabla_{J}(\lambda)$ admitting the character of a module with a $p$-filtration. 

\begin{theorem} If $G$ satisfies the Tilting Module Conjecture, then $\nabla_{J}(\lambda)$ has the character of a module admitting a good $(p,r)$-filtration for all $\lambda\in X^{+}$, $J\subseteq \Delta$ and 
$r\geq 1$. 
\end{theorem} 

\begin{proof} Let $r\geq 1$. First recall that if the TMC holds for $G$, then 
$\text{Hom}_{G_{r}}(\widehat{Q}_{r}(\sigma),\nabla(\lambda))$ has a good filtration for all $\sigma\in X_{r}$ and $\lambda\in X^{+}$ \cite[Theorem 9.2.3]{KN}. 

Next observe that  if $M$ is a finite-dimensional $G$-module, then 
$$\text{ch }M=\sum_{\sigma\in X_{r}}\text{ch }L(\sigma)\otimes \text{ch }\text{Hom}_{G_{r}}(\widehat{Q}_r(\sigma),M).$$ 
This can be proved via induction on the composition length of $M$ and using the fact that for $\lambda\in X^{+}$
with $L(\la) = L(\la_0) \otimes L(\la_1)^{(r)}$, $\la_0 \in X_r$ and $\la_1 \in X^+$, 
 the expression $\Hom_{G_r}(\widehat{Q}_r(\sigma), L(\la)) \cong \Hom_{G_r}(\widehat{Q}_r(\sigma), L(\la_0))\otimes L(\la)^{(r)}$ vanishes unless $\sigma = \la_0$, in which case it is $L(\la_1)^{(r)}.$ 

Therefore, by using these facts, it follows that $M=\nabla(\lambda)$ has the character of a module with a good $(p,r)$-filtration. 
The statement for $\nabla_{J}(\lambda)$ where $J\subseteq \Delta$ follows by Proposition~\ref{P:tensor}. 
\end{proof} 

The statement of the prior theorem for $J=\Delta$ was also observed by Kildetoft and conveyed to the third author via a private correspondence. 


\section{Using Extensions Between Simple Modules to Generate Counterexamples}\label{S:methods}

In this section, we show that the structure of the $\text{Ext}^{1}$ between two simple $G_{1}$-modules is a key ingredient to the validity of the TMC. In the process of our analysis 
we present several related methods for constructing counterexamples to the TMC.  

\subsection{First Method}\label{S:1st}  We begin by making the following observation. 

\begin{prop} \label{P:FirstConstruction}
Let $\la, \mu \in X_1$ with $\la \neq \mu$.  If the Tilting Module Conjecture holds, then $\Ext_{G_1}^1(L(\la), L(\mu)) $ is a $G$-submodule of  $\Hom_{G_1}(Q_1(\la), Q_1(\mu)).$ 
\end{prop}

\begin{proof} The TMC implies that $Q_1(\sigma)$ can be lifted to $G$-modules for all $\sigma\in X_{1}$ and are tilting modules. In particular, $\Hom_{G_1}(Q_1(\la), Q_1(\mu))$ has a $G$-module structure. Consider  
$$0 \to L(\mu) \to Q_1(\mu) \to Q_1(\mu)/L(\mu) \to 0.$$
Using $\la \neq \mu$ one immediately obtains 
$$\Hom_{G_1}(L(\la),  Q_1(\mu)/L(\mu)
) \cong \Ext_{G_1}^1(L(\la), L(\mu)) .$$
Next we use 
$$ 0 \to \mbox{rad}(Q_1(\la)) \to Q_1(\la) \to L(\la) \to 0,$$
to obtain that 
$$  \Ext_{G_1}^1(L(\la), L(\mu)) \cong \Hom_{G_1}(L(\la),  Q_1(\mu)/L(\mu)) \hookrightarrow  \Hom_{G_1}(Q_1(\la),  Q_1(\mu)/L(\mu)).$$
Finally, using the first short exact sequence and the operator $\Hom_{G_1}(Q_1(\la), - )$ yields
 $$  \Ext_{G_1}^1(L(\la), L(\mu))  \hookrightarrow  \Hom_{G_1}(Q_1(\la),  Q_1(\mu)/L(\mu))\cong \Hom_{G_1}(Q_1(\la), Q_1(\mu)).$$
\end{proof}

Note that the TMC implies that  $\Hom_{G_1}(Q_1(\la), Q_1(\mu))^{(-1)}$ is a tilting module \cite{KN}. Moreover, the weights appearing in $\Hom_{G_1}(Q_1(\la), Q_1(\mu))$ are less than or equal to $2(p-1)\rho- \la + \omega_0\mu.$ One obtains immediately the following theorem. 

\begin{theorem}\label{T:FirstConstruction} Let $\la, \mu \in X_1$ with $\la \neq \mu$.  Assume the Tilting Module Conjecture holds.
\begin{itemize} 
\item[(a)] Then
$\Ext_{G_1}^1(L(\la), L(\mu))^{(-1)} $ is a $G$-submodule of some tilting module whose weights $\gamma$ satisfy $p \gamma \leq 2(p-1)\rho- \la + w_0\mu$. 
\item[(b)] If $L(\nu) \hookrightarrow \Ext_{G_1}^1(L(\la), L(\mu))^{(-1)},$ 
then $L(\nu)$ has to be a submodule of a Weyl module $\Delta(\ga)$, with $p \gamma \leq 2(p-1)\rho- \la + w_0\mu$.
\end{itemize} 
\end{theorem}

\begin{proof} (a) This follows by the observation stated before the theorem and Proposition~\ref{P:FirstConstruction}. 

(b) From part (a), we know that if the TMC holds, then $L(\nu)$ has to appear in the socle  of some indecomposable tilting module $T(\ga)$, with $p \gamma \leq 2(p-1)\rho- \la + w_0\mu.$  The assertion follows from the fact that a tilting module has a Weyl filtration.
\end{proof}

\subsection{Counterexample Revisited for $\Phi=\rm{G}_{2}$}
 Let $G$ be of type $\rm{G}_2$ and $p=2.$ It was shown in \cite{BNPS20} that the TMC fails in this case. We will make use of the above set-up and give  a new argument. 
 
Let $\la = 0$ and $\mu= \omega_2$. According to  \cite[Prop. 5.2]{Jan91}, 
$$\Ext_{G_1}^1(k, L(\omega_2))^{(-1)} \cong  \Ext_{G_1}^1(k, \nabla(\omega_2))^{(-1)} \cong \nabla(\omega_1).$$ 
Moreover, $2\rho - \la +w_0\mu = 2\rho - \omega_2 = 2\omega_1 + \omega_2.$ The only weights $\ga$ with $2 \ga \leq 2\omega_1 + \omega_2$ are $0$, $\omega_1$, and $\omega_2$. The corresponding Weyl modules $\Delta(0)$, $\Delta(\omega_1)$, and $\Delta(\omega_2)$ have simple socles $k$, $k$, and $L(\omega_2)$, respectively. If the TMC held, this would contradict Theorem~\ref{T:FirstConstruction}(b). 
Thus the Tilting Module Conjecture fails for $\rm{G}_2$ and $p=2$.

\subsection{Second Method}\label{S:2nd} If we assume that $\Hom_{G_1}(Q_1(\la),\nabla(\mu))$ vanishes,  one obtains the following modifications of Proposition~\ref{P:FirstConstruction} and 
Theorem~\ref{T:FirstConstruction}.

\begin{prop}\label{P:SecondConstruction} Let $\la, \mu \in X_1$ with $\la \neq \mu$. If the Tilting Module Conjecture holds and $\Hom_{G_1}(Q_1(\la),\nabla(\mu))= 0$, then $\Ext_{G_1}^1(L(\la), \nabla(\mu)) $ is a $G$-submodule of  $\Hom_{G_1}(Q_1(\la), Q_1(\mu)).$ 
\end{prop} 

\begin{theorem}\label{T:SecondConstruction} Let $\la, \mu \in X_1$ with $\la \neq \mu$.  Assume the Tilting Module Conjecture holds and  $\Hom_{G_1}(Q_1(\la),\nabla(\mu))= 0$.
\begin{itemize} 
\item[(a)] Then $\Ext_{G_1}^1(L(\la), \nabla(\mu))^{(-1)} $ is a $G$-submodule of some tilting module $T(\ga)$, with $p \gamma \leq 2(p-1)\rho- \la + w_0\mu$.
\item[(b)] If
$L(\nu) \hookrightarrow \Ext_{G_1}^1(L(\la), \nabla(\mu))^{(-1)},$ 
then $L(\nu)$ has to be a submodule of a Weyl module $\Delta(\ga)$, with $p \gamma \leq 2(p-1)\rho- \la + w_0\mu$.
\end{itemize} 
\end{theorem}

The next result enables the employment of Theorem~\ref{T:SecondConstruction}. This strategy will be used later to produce a counterexample to the TMC in 
type $\rm{B}_n$.

\begin{prop}\label{P:SecondConstruction2} Let $\la, \mu \in X_1$ satisfy the equation
$\la + p \omega_i = \mu + \alpha_i,$
where $\alpha_i$ denotes a simple root and $\omega_i$ the corresponding fundamental weight. In addition, assume that $\langle \la, \alpha_i^{\vee}\rangle =0.$ Then there exists a $G$-module monomorphism
$$\nabla(\omega_i) \hookrightarrow   \Ext_{G_1}^1(L(\la), \nabla(\mu))^{(-1)}. $$
\end{prop}

\begin{proof} In computing $G_1$-extensions, one has  (cf. \cite[Lemma II.12.8]{rags})
$$\Ext_{G_1}^1(L(\la), \nabla(\mu))^{(-1)} \cong \ind_{B}^G[\Ext_{B_1}^1(L(\la), \mu)^{(-1)}].$$
There exists a short exact sequence of $B$-modules
$$0 \to \mu \to   k[U_1] \otimes \mu \to k[U_1]/k \otimes \mu \to 0,$$
which yields the  exact sequence 
$$ 0 \to \Hom_{B_1}(L(\la),k[U_1] \otimes \mu) \to      \Hom_{B_1}(L(\la),k[U_1]/k \otimes \mu) \to \Ext_{B_1}^1(L(\la), \mu) \to  0.$$
From weight considerations one obtains a $B$-module injection
$$p \omega_i  \cong \Hom_{B_1}(L(\la),\mu + \alpha_i)\hookrightarrow \Hom_{B_1}(L(\la),k[U_1]/k \otimes \mu).$$
Note that $\langle \la, \alpha_i^{\vee}\rangle =0$ implies that $\lambda - \alpha_i$ is not a weight of $L(\lambda).$ This implies that $p\omega_i$ is not a weight of   $\Hom_{B_1}(L(\la),k[U_1] \otimes \mu).$ 
Consequently, there is an injection
$$p \omega_i  \hookrightarrow \Ext_{B_1}^1(L(\la), \mu).$$
Since induction is left-exact, one obtains a $G$-module monomorphism: 
$$\nabla(\omega_i) \hookrightarrow \Ext_{G_1}^1(L(\la), \nabla(\mu))^{(-1)}.$$
\end{proof}

\subsection{Third Method}\label{S:3rd}
For a restricted weight $\la \in X_1$ and a dominant weight $\mu$, we will look at $\Hom_{G_1}({Q}_1(\la), \nabla(\mu)).$    If the TMC holds, then it follows from \cite[Theorem 9.2.3]{KN}
 that $\Hom_{G_1}({Q}_1(\la), \nabla(\mu))^{(-1)}$ has a good filtration. The idea now is to find appropriate weights $\la$ and $\mu$ that violate this last statement. We will use this technique to produce counterexamples for groups of type $\rm{B}_3$ with $p = 2$ and type $\rm{C}_3$ with $p=3.$


\section{Low Rank Counterexamples}\label{S:Low}

\subsection{Summary of Results:} In this section we will show that the TMC fails and Jantzen's Question has a negative answer for $G$ when $\Phi=\rm{B}_{3}$, $\rm{C}_{3}$ and $\rm{D}_{4}$. In particular, the following theorem will be proved. 

\begin{theorem}\label{T:small-rank} Let $G$ be a simple algebraic group. Then $T(2(p-1)\rho)$ is not isomorphic to $\widehat{Q}_1(0)$ as a $G_1T$-module in the following cases: 
\begin{itemize} 
\item[(a)] $\Phi=\rm{B}_{3}$, $p=2;$
\item[(b)]  $\Phi=\rm{C}_{3}$, $p=3;$
\item[(c)] $\Phi=\rm{D}_{4}$, $p=2.$
\end{itemize} 
In the cases (a)-(c), there exists an induced module $\nabla(\lambda)$, $\lambda\in X^{+}$, that does not admit a good $p$-filtration.
\end{theorem} 

\subsection{Jantzen Filtration and Jantzen Sum Formula}\label{S:Ja-Fil} Several of the important calculations that we will use are derived via the Jantzen filtration (cf. \cite[Prop. II.8.19]{rags}) as described below. 

For each $\la \in X^+$, there is a filtration of $G$-modules
$$
\Delta(\la) = \Delta(\la)^0 \supseteq \Delta(\la)^1 \supseteq \Delta(\la)^2 \supseteq \cdots
$$
such that $\text{rad}_{G}\Delta(\lambda)=\Delta(\la)^{1}$ and $\Delta(\la)/\Delta(\la)^1 \cong L(\la)$. Moreover,  
$$
\sum_{i > 0} \operatorname{ch} \Delta(\la)^i = \sum_{\a \in \Phi^+}\sum_{0 < mp < \langle\la + \rho,\a^{\vee}\rangle} \nu_p(mp)\chi(s_{\a,mp}\cdot \la).
$$

\subsection{Some Character Data} For $\Phi=\rm{B}_{n}$ when $p=2$, we record the following information about the structure of various representations. 

\begin{lemma}\label{L:Data}
Let $G$ be of type $\rm{B}_n$, $n \geq 3,$ and $p=2.$
\begin{itemize}
\item[(a)] The composition factors of $\nabla(\omega_1)$ are $L(\omega_1)$ and the trivial module, each appearing once. In particular, the head of $\nabla(\omega_1)$ consists of the trivial module.
\item[(b)] For $n$ odd, the module $\nabla(\omega_2)$   is uniserial with the three composition factors of $L(\omega_2)$, $L(\omega_1)$ and the trivial module, listed from bottom to top. Its head consists of the trivial module.
\item[(c)] For $n$ even, the module $\nabla(\omega_2)$   has four composition factors: $L(\omega_2)$, $L(\omega_1)$, and the trivial module with multiplicity two. Its head consists of just one copy of the trivial module.
\item[(d)] For $n = 3$, the dominant weights appearing in the simple module $L(\omega_1+\omega_2)$  are $\omega_1 + \omega_2$ with multiplicity one, $2 \omega_3$ with multiplicity two, and $\omega_1$ with multiplicity four.
\end{itemize}
\end{lemma} 

\begin{proof} Note that $\nabla(\omega_2)$ is the dual of the adjoint representation.  Claims (a) through (c) can be found in \cite[Proposition 6.9(a)]{Jan91} and its proof.

For part (d), one can make use of the special isogenies between types $\rm{B}_n$ and $\rm{C}_n$ that exist for $p=2$ \cite[I.3]{DS}. As a result, one obtains a sharpened version of the Steinberg tensor product theorem
\cite[I.4.2 (3)]{DS}. The character of the $\rm{B}_3$-module $L(\omega_1+ \omega_2)$ can be obtained via the character of the  $\rm{C}_3$-module $L(\omega_1 + \omega_2),$ which is identical to the  $\rm{C}_3$-module $\nabla(\omega_1 + \omega_2).$ 
\end{proof} 


\subsection{$\Phi=\rm{B}_3$ and $p=2$} In \cite{BNPS21} it was shown that the TMC holds for a group of type $\rm{B}_2$ and all primes. In this section, the third method that was introduced in Section \ref{S:3rd} is used to show that the TMC fails for a group of type $\rm{B}_3$ and $p=2.$  

\subsubsection{Tilting Module Conjecture}\label{S:TMCB3} Assume that $T(2\rho)\mid_{G_{1}T}\cong\widehat{Q}_1(0)$.  We will show that $\Hom_{G_1}(\widehat{Q}_1(0), \nabla(2 \omega_2))^{(-1)}$ does not afford a good filtration, thereby obtaining a contradiction to \cite[Theorem 9.2.3]{KN}.

The dominant weights less than or equal to $2\om_2$ are  
$2 \omega_2,$ $\om_1+ 2 \om_3,$ $\om_1+ \om_2,$ $2\om_3,$ $2\om_1,$ $\om_2,$ $\om_1,$ and $0.$ All of these are linked to $2\om_2$ and also appear in $\nabla(2 \omega_2). $ Note that $\omega_3$ is minuscule.
Lemma \ref{L:Data}
provides therefore sufficient data to determine the characters of all simple modules with highest weights from the above list as well as all the decomposition numbers for all $\nabla(\sigma)$ with weights from the list. Alternatively, one could also refer to the tables in \cite{Lue}. For our argument we will make use of the Jantzen filtration for the Weyl module $\Delta(2\om_2).$ The following table lists  the multiplicities of each composition factor of  $\Delta(2 \om_2)$ as well as the multiplicities of each simple factor in the Jantzen sum formula. We include only weights with positive multiplicity.
\vskip .25cm 
\begin{center} 
\begin{tabular}{|c |c |c|}
\hline
$\la$ & $\left[ \Delta(2 \om_2):L(\la)\right]$ & $ \left[\sum_{i>0} \text{ch } \Delta(2 \om_2)^i:L(\la)\right]$ \\
 
\hline \hline
$2\om_2$ & 1 &0\\
\hline
$\om_1+2 \om_3$ & 1&1\\
\hline
$\om_1+\om_2$ &1&2\\
\hline
$2\om_1$& 2&2\\
\hline
$\om_2$ &2&4  \\
\hline
 $0$& 2&2\\
 \hline
 
 \end{tabular}
 \end{center}
\vskip .25cm 

The above table shows that any composition factor with highest weight $\om_1+2\om_3$, $2\om_1,$ or $0$ has to appear in the second highest layer of the Jantzen filtration, that is, in $\Delta(2\omega_2)^1/\Delta(2\omega_2)^2$. Hence, only composition factors with highest weight $\om_2$ or $\om_1+\om_2$ appear in $\Delta(2 \om_2)^2.$  Recall the $\tau$-functor as defined in Section \ref{S:notation}.  Set $S=^{\tau}\!\!\!\left(\Delta(2\om_2)/ \Delta (2\om_2)^2\right)$. Then one obtains the short exact sequence via the $\tau$-functor:
$$0 \to S \to \nabla(2\om_2) \to  ^{\tau}\!\!\!\left(\Delta (2\om_2)^2\right)\to 0.$$
It follows that 
$$\Hom_{G_1}(\widehat{Q}_1(0), \nabla(2 \omega_2)) \cong \Hom_{G_1}(\widehat{Q}_1(0), S).$$
In addition, there exists an embedding $\nabla(\om_2)^{(1)} \hookrightarrow \nabla ( 2\om_2).$
Since none of the composition factors of $\nabla(\om_2)^{(1)}$ appear in $^{\tau}\!\!\left(\Delta (2\om_2)^2\right)$ one obtains an embedding  $\nabla(\om_2)^{(1)} \hookrightarrow S$.  Lemma \ref{L:Data}(b) now yields embeddings $$L(\omega_1)^{(1)}\hookrightarrow \nabla( \omega_1)^{(1)} \hookrightarrow S/L(\omega_2)^{(1)}=^{\tau}\!\!\!\left(\Delta(2\om_2)^1/ \Delta (2\om_2)^2\right).$$
 Note that the layers in the Jantzen filtration are ${\tau}$-invariant (cf. \cite[II.8.19(3)]{rags}). One obtains a projection $\pi: S\twoheadrightarrow L(\om_1)^{(1)}.$
Next  we define $Q$ via 
$$0 \to \nabla(\omega_2)^{(1)} \to S \to Q \to 0.$$
Since  $\nabla(\omega_2)$ has a simple head isomorphic to the trivial module, $\pi$ has to factor through $Q$.  Therefore,  
both $Q$ and   $ \Hom_{G_1} (\widehat{Q}_1(0), Q)$ map onto $L(\om_1)^{(1)}.$
Since both $[S:L(\om_1)^{(1)}]$ and $[S:k]$ are at most two, one concludes, via subtraction of the character of $\nabla(\om_2)^{(1)},$ that $[Q: L(\om_1)^{(1)}]=1$ and that $[Q: k]\leq 1.$ No other composition factor of $Q$ can contribute to $ \Hom_{G_1} (\widehat{Q}_1(0), Q) .$
The character of $ \Hom_{G_1} (\widehat{Q}_1(0), Q)^{(-1)} $  is therefore either equal to the character of $L(\om_1)$ or to the character of $L(\om_1)$ together with a trivial character.  Since $ \Hom_{G_1} (\widehat{Q}_1(0), Q)^{(-1)} $ maps onto $L(\om_1)$ and  $\nabla(\omega_1) \neq L(\om_1)$ one concludes that $ \Hom_{G_1} (\widehat{Q}_1(0), Q)^{(-1)}$ cannot have a good filtration.
\\\\
 On the other hand, by looking at
$$0 \to \Hom_{G_1} (\widehat{Q}_1(0), \nabla(\omega_2)^{(1)}) \to \Hom_{G_1} (\widehat{Q}_1(0), S)  \to \Hom_{G_1} (\widehat{Q}_1(0), Q) \to 0$$
which is equivalent to 
$$0 \to  \nabla(\omega_2)^{(1)} \to \Hom_{G_1} (\widehat{Q}_1(0), \nabla(2\om_2))  \to \Hom_{G_1} (\widehat{Q}_1(0), Q)\to 0,
$$
one concludes that  $\Hom_{G_1} (\widehat{Q}_1(0), \nabla(2\om_2))^{(-1)}$ does not afford a good filtration.

\subsubsection{Jantzen's Question}\label{S:JQB3} We show that the induced module $\nabla(2 \om_2)$ does not afford a good $p$-filtration. Suppose that $\nabla(2 \om_2)$ has a good $p$-filtration. From the data given in the table of Section \ref{S:TMCB3}, one concludes that the factors $\nabla(\om_2)^{(1)}$ and $\nabla(\om_1)^{(1)}$ each have to appear once in such a filtration. The first, $\nabla(\om_2)^{(1)},$ appears at the very bottom of any good $p$-filtration.  Then $V := \nabla(2\om_2)/\nabla(\om_2)^{(1)}$ also has as good $p$-filtration with one of the factors being $\nabla(\om_1)^{(1)}$.  The module$\nabla(\om_1)^{(1)}$  has two composition factors: the trivial module $k$ and $L(\om_1)^{(1)}$. Further, we would have $[V : L(\om_1)^{(1)}] = 1 = [V : k]$. By the nature of $\nabla(\om_1)^{(1)}$, in the radical series for $V$,  the copy of $k$ must appear higher than the $L(\om_1)^{(1)}$.  
This means that the $k$ must appear above the $L(\om_1)^{(1)}$ in any composition series for $V$ (as, in general, a composition series is a refinement of the radical series).  
However, the argument in the previous section shows that there exists a composition series of $V$ in which a composition factor of the form $L(\om_1)^{(1)}$ appears higher than the factor isomorphic to the trivial module; a contradiction.

\subsubsection{} The methods of Section \ref{S:methods} suggest a connection between the validity of the TMC and good filtrations on $G_1$-extension groups. With an eye towards a more precise connection (that will be discussed further in Section \ref{S:questions}), we note that, from \cite[Table II.2.5(a) on page 2632]{DS} or \cite[Prop. 6.9]{Jan91}, $$\Ext_{G_1}^1(k, L(\om_2))^{(-1)} \cong \nabla(\om_1),$$ which is not tilting.


\subsection{$\Phi=\rm{C}_3$ and $p=3$} We again employ the method introduced in Section \ref{S:3rd} to show that the Tillting Module Conjecture fails in this case. 

\subsubsection{Tilting Module Conjecture}\label{S:TMCC3} By using the Jantzen filtration one obtains the following tables. 
\vskip .25cm 
\begin{center} 
\begin{tabular}{|c|c|}
\hline
$\la$ & $\sum_{i>0} \text{ch } \Delta(\la)^i$\\
\hline
\hline
$(0,0,0)$ & $\emptyset$\\
\hline
$(0,1,0)$ & $\chi(0,0,0)$\\
\hline
$(1,0,1)$ & $\chi(0,1,0) - \chi(0,0,0)$\\
\hline
$(0,0,2)$ & $\emptyset$\\
\hline
$(1,1,1)$ &  $\chi(0,0,2) +2\times \chi(1,0,1)  - \chi(0,1,0) + \chi(0,0,0)$\\
\hline
 $(0,3,0)$ & $\chi(1,1,1) - \chi(0,0,2) - \chi(1,0,1)+ 2\times \chi(0,1,0) - \chi(0,0,0)$ \\
 \hline
 $(2,0,2)$ & $2 \times \chi(1,1,1) + \chi(0,0,2) - 2 \times \chi(1,0,1) + \chi(0,1,0) - \chi(0,0,0)$ \\
 \hline
 $(2,1,2)$ & $\chi(2,0,2) +2 \times \chi(0,3,0)-\chi(0,0,2)+ \chi(1,0,1)- 2 \times \chi(0,1,0)+ \chi(0,0,0)$ \\
 \hline
 \end{tabular}
\end{center}
\vskip .25cm 
\begin{center}

\begin{tabular}{|c|c|||c|c|}
\hline
$\la$ & $\left[\sum_{i>0} \text{ch } \Delta(\la)^i:L(0,0,0)\right]$  & $\la$ & $\left[\sum_{i>0} \text{ch } \Delta(\la)^i:L(0,3,0)\right]$ \\
\hline
\hline
$(0,0,0)$ & $0$ & $(2,0,2)$ & $0$ \\        
\hline
$(0,1,0)$ & $1$  &  $(2,1,2)$ & $2$ \\
\hline
$(1,0,1)$ & $0$   &           &  \\
\hline
$(0,0,2)$ & $0$    &           &\\
\hline
$(1,1,1)$ &  $0$    &            &\\
\hline
 $(0,3,0)$ & $1$     &             &\\
 \hline
 $(2,0,2)$ & $0$    &             &\\
 \hline
 $(2,1,2)$ & $1$      &             &\\
 \hline
 \end{tabular}
\end{center}
\vskip .25cm 

Assume that $T(4 \rho) \mid_{G_1T}\cong  \widehat{Q}_1(0,0,0)$. From the tables above, we conclude that 
$$[\sum_{i>0} \text{ch } \Delta(2,1,2)^i:L(0,0,0)] = [\Delta(2,1,2): L(0,0,0)] =1,$$
while 
$$ [\Delta(2,1,2)^2:L(0,0,0)]=0.$$
Note that the only non-restricted composition factor of $\Delta(2,1,2)$ has highest weight $(0,3,0)$ and that 
$\left[\sum_{i>0} \text{ch } \Delta(2,1,2)^i:L(0,3,0)\right] = 2.$ There are two possibilities:
\vskip .25cm 
\noindent 
{\bf Case 1}: $ [ \Delta(2,1,2)^2:L(0,3,0)]=0.$
\vskip .15cm 
This implies that $[\Delta(2,1,2): L(0,3,0)]=2$ and therefore 
\begin{eqnarray*}
\text{ch}\Hom_{G_1}(T(4 \rho): \nabla(2,1,2))^{(-1)} &=& \text{ch}\Hom_{G_1}(\hQ_1(0,0,0), \nabla(2,1,2))^{(-1)}\\
& =&
\text{ch}\Hom_{G_1}(\hQ_1(0,0,0), \Delta(2,1,2))^{(-1)} \\
&=&
2\cdot\text{ch } L(0,1,0) + \text{ch } L(0,0,0).
\end{eqnarray*}
Since $\nabla(\omega_2)$ has composition factors $L(\omega_2)$ and $k$, $\text{ch}\Hom_{G_1}(T(4 \rho): \nabla(2,1,2))^{(-1)}$ cannot be the character of a module with a good filtration. Thus we obtain a contradiction to \cite[Theorem 9.2.3]{KN}.
\vskip .25cm 
\noindent 
{\bf Case 2}: $ [ \Delta(2,1,2)^2:L(0,3,0)]=1.$
\vskip .15cm 
This implies that $[\Delta(2,1,2): L(0,3,0)]=1.$ Define $Q$ via the exact sequence
$$0 \to \Delta(2,1,2)^2 \to \Delta(2,1,2)\to Q \to 0,$$
which gives rise to the short exact sequence 
$$ 0 \to \Hom_{G_1}(\hQ_1(0,0,0), \Delta(2,1,2)^2)^{(-1)} \to  \Hom_{G_1}(\hQ_1(0,0,0), \Delta(2,1,2))^{(-1)}$$
$$\to  \Hom_{G_1}(\hQ_1(0,0,0), Q)^{(-1)} \to 0.$$
This sequence is equivalent to 
$$0 \to L(0,1,0) \to \Hom_{G_1}(\hQ_1(0,0,0), \Delta(2,1,2))^{(-1)} \to L(0,0,0) \to 0.$$
Its dual version is
$$0 \to L(0,0,0) \to \Hom_{G_1}(\hQ_1(0,0,0), \nabla(2,1,2))^{(-1)} \to L(0,1,0) \to 0.$$
Again we obtain a contradiction to \cite[Theorem 9.2.3]{KN}.

\subsubsection{Jantzen's Question}\label{S:JQC3} We show that the induced module $\nabla(2,1,2)$  does not afford a good $p$-filtration. 
Assume that $\nabla(2,1,2)$ has a good $p$-filtration. 
The weight $3\omega_2$ is maximal among the dominant weights of the form $3\ga$ that appear in $\nabla(2,1,2).$
From the data given in the  tables in Section \ref{S:TMCC3}, one can see that $\nabla(\om_2)^{(1)}$ has to appear at least once, possibly twice, in any good $p$-filtration of $\nabla(2,1,2).$ 
Note that $\nabla(\om_2)$ has two composition factors with highest weights $\om_2$ and $0$. 
Correspondingly, $\nabla(\om_2)^{(1)}$ will have factors $L(\om_2)^{(1)}$ (socle) and $k$ (head).  
As argued in Section \ref{S:JQB3} for type $\rm{B}_3$,  for $\nabla(2,1,2)$ to admit a good $p$-filtration, 
every appearance of $L(\om_2)^{(1)}$ must lie below an occurence of $k$ in its radical series, and, hence, in any composition series. 
Consider Cases 1 and 2 as in the argument of Section \ref{S:TMCC3}.  In the first case, $\nabla(2,1,2)$ would have two copies of $L(\om_2)^{(1)}$, but only a single copy of $k$; a
clear contradiction.  In Case 2, $k$ appears in $\Delta(2,1,2)^1$, while $L(\om_2)^{(1)}$ appears in $\Delta(2,1,2)^2$ (and only once in $\Delta(2,1,2)$). Dualizing, there exists a composition series of 
$\nabla(2,1,2)$ for which a composition factor isomorphic to $L(\om_2)^{(1)}$ appears higher than any trivial module. With this contradiction, we conclude that  $\nabla(2,1,2)$ does not afford a good $p$-filtration.

\subsubsection{} 
From the table in Section \ref{S:TMCC3}, one obtains that $[\nabla(1,1,1):k]=0$ and therefore $\Hom_{G_1}(k,  \nabla(1,1,1)/L(1,1,1))=0.$ 
From \cite[Prop. 4.1 and Section 4.2]{Jan91}, one has 
$$
\Ext_{G_1}^1(k, L(1,1,1))^{(-1)} \hookrightarrow \Ext_{G_1}^1(k, \nabla(1,1,1))^{(1)} \cong \nabla(\om_2).
$$
Furthermore, from \cite[Prop. 4.5]{Jan91}, 
$$L(\om_2) \hookrightarrow \Ext_{G_1}^1(k, L(1,1,1))^{(-1)}.$$ It follows that    $\Ext_{G_1}^1(k, L(1,1,1))^{(-1)}$ is not tilting.   

\begin{remark} More generally, consider the case of type $\rm{C}_p$ for $p \geq 3$.  The module $\nabla(\om_2)$ is not simple.   It has two composition factors: $L(\om_2)$ and the trivial module $k$.  It follows from \cite[Prop. 4.1, Section 4.2, and Prop. 4.5]{Jan91} that $\Ext_{G_1}^1(k,\nabla(p\omega_2 - \a_2))^{(-1)} \cong \nabla(\om_2)$ and one has an exact sequence
$$
0 \to (\nabla(p\om_2 - \a_2)/L(p\om_2 - \a_2))^{G_1} \to \Ext_{G_1}^1(k,L(p\om_2 - \a_2)) \to \Ext_{G_1}^1(k,\nabla(p\om_2 - \a_2)),
$$
where the image of the last map contains $L(\om_2)^{(1)}$.  By direct computation, the only dominant weight $\ga$ for which $p\ga$ is a weight of $\nabla(p\om_2 - \a_2)$ is $\ga = 0$. As such, $(\nabla(p\om_2 - \a_2)/L(p\om_2 - \a_2))^{G_1}$ is a (possibly empty) sum of trivial modules.  If this term vanishes, we would see that  $\Ext_{G_1}^1(k,L(p\om_2 - \a_2))^{(-1)}$ is not a tilting module, but we have been unable to confirm this.  We speculate that this is the case and that the TMC fails in type $\rm{C}_p$ for $p \geq 3$.  
\end{remark}


\subsection{$\Phi=\rm{D}_4$ and $p=2$}

In this section the first method introduced in Section \ref{S:1st} is used to show that the TMC fails for $\rm{D}_4$ and $p=2$. 

\subsubsection{Tilting Module Conjecture}\label{S:TMCD4} We demonstrate that the tilting module $T(2\rho)$ is not isomorphic to $\widehat{Q}_1(0)$ as a $G_1T$-module.
According to \cite[3.4(b), p. 2659]{DS} (cf. also \cite{Sin94}), 
\begin{equation}
\Ext_{G_1}^1(k, L(\omega_1 +\omega_3+ \omega_4)) \cong k  \oplus L(\omega_2)^{(1)}.
\end{equation}
Assume that the TMC holds. Then, according to Proposition~\ref{P:FirstConstruction}, 
\begin{equation}
L(\omega_2)^{(1)}\hookrightarrow \Ext_{G_1}^1(L(0), L(\omega_1 +\omega_3+ \omega_4)) \hookrightarrow \Hom_{G_1}(\hQ_1(0), \hQ_1(\omega_1 + \omega_3+ \omega_4)).
\end{equation}
Moreover, according to Theorem~\ref{T:FirstConstruction}(b), $L(\omega_2)$ must appear in the socle of a Weyl module $\Delta(\ga)$ with highest weight less than or equal to $2\rho - (\omega_1+\omega_3+\omega_4) = \rho +  \omega_2$.  Any weight $\gamma$ that is linked to $\omega_2$ and satisfies $2 \gamma \leq \rho + \omega_2$ is contained in the following list: 
$$\omega_1 + \omega_3+ \omega_4, 2\omega_1, 2\omega_3, 2\omega_4, \omega_2, 0.$$
The module $\Delta(\omega_2)$ is the dual of the adjoint representation. Its radical consists of two copies of the trivial module. The modules $\Delta(2\omega_i)$ with $i \in \{1, 3, 4\}$ are uniserial with the trivial module as their socle and a middle consisting of $L(\omega_2)$ \cite[Lemma 4.3]{Sin94}. 

Next we observe that $L(\omega_2)$ does not appear in the socle of $\Delta(\omega_1 + \omega_3 + \omega_4)$. We embed $\Delta(\omega_1+\omega_3+\omega_4)$ in $\Delta(\omega_1) \otimes\Delta(\omega_3+\omega_4)$ and show that 
$$\Hom_G(L(\omega_2), \Delta(\omega_1) \otimes \Delta(\omega_3+\omega_4))\cong \Hom_G( \nabla(\omega_3+\omega_4), L(\omega_2) \otimes \Delta(\omega_1))=0.$$ 
Note that $\Delta(\omega_1) = L(\omega_1)$ while $\nabla(\omega_3+\omega_4)$ has length two with simple head $L(\omega_1).$ Moreover, a straightforward character calculation shows that $L(\omega_2) \otimes L(\omega_1)$ has composition factors $L(\omega_1 + \omega_2) $ and $L(\omega_3 + \omega_4).$ Since  $L(\omega_1)$ does not appear as a factor, the socle of $\Delta(\omega_1 + \omega_3 + \omega_4)$ does not contain $L(\omega_{2})$, giving a contradiction to Theorem~\ref{T:FirstConstruction}.

Thus, at least one of $\widehat{Q}_1(0)$ or $\widehat{Q}_1(\omega_1 + \omega_3+ \omega_4))$ is not a tilting module. 
From character data (particularly, that the dominant weights of $\nabla(\omega_2)$ are only $\omega_2$ and 0), $T(\rho+\omega_2) \mid_{G_1T} \cong \widehat{Q}_1(\omega_1 + \omega_3 + \omega_4).$ Therefore, $T(2 \rho)$ cannot be isomorphic to $\widehat{Q}_1(0)$, as a $G_1T$-module.

\subsubsection{Jantzen's Question} \label{S:JQD4} It was shown in \cite[5.6]{BNPS19} that Conjecture~\ref{C:stla} holds in this case. Moreover, from \cite{So}  we know that an affirmative answer to Jantzen's Question (i.e., 
Question~\ref{Q:JQ}) together with Conjecture~\ref{C:stla} implies that the TMC holds.  Since the TMC fails here, one concludes that there exists a dominant weight for which Jantzen's Question does not have a positive answer.

\section{Proof of the Main Theorem} 

\subsection{Proof}  The following is a list of  groups and their root systems for which the small rank groups of Theorem \ref{T:small-rank} appear naturally as Levi subgroups 
\begin{itemize} 
\item $\rm{B}_3:$ appears in $\rm{B}_n,$ $n \geq 3,$ and in $\rm{F}_4;$
 \item $\rm{C}_3:$ appears in $\rm{C}_n,$ $n \geq 3,$ and in $\rm{F}_4;$
 \item $\rm{D}_4:$ appears in $\rm{D}_n,$ $n \geq 4,$ and in $\rm{E}_n$, $n=6, 7, 8;$
 \end{itemize}
 This allows one to apply the contrapositive of Theorem \ref{T:TMCLevi} and extend the statement concerning the TMC of Theorems \ref{T:small-rank}
 to the root systems listed in Section \ref{S:main}, for their respective primes. Similarly, the contrapositive of Theorem \ref{T:JQLevi} together with Sections \ref{S:JQB3},  \ref{S:JQC3}, and \ref{S:JQD4} immediately 
verify the part of the Main Theorem concerning JQ. 

The following section gives a more precise statement concerning the failure of the TMC. 

\subsection{Statement with Weights} Keeping track of the weights in Theorem \ref{T:small-rank} as well as including the  counterexamples  of Theorems \ref{T:Bn} and \ref{T:C3-more}, 
provides for a more extensive list of cases in which the TMC  fails. We also include the original counterexample in \cite{BNPS20}.

\begin{theorem}\label{T:main2} Let $G$ be a simple algebraic group over an algebraically closed field of characteristic $p>0$ with underlying root system $\Phi$ and $\la$ a $p$-restricted weight. Then
$$T((p-1)\rho+\lambda)|_{G_{1}T} \neq \widehat{Q}_{1}((p-1)\rho+w_{0}\lambda),$$
provided the  triple $(\Phi, p, \la)$ appears in the following list
\begin{itemize} 
\item $\Phi=\rm{B}_{n}$, $n\geq 3$, $p=2,$ and
\begin{itemize}
\item[(i)] $\langle \la, \alpha_i^{\vee} \rangle = p-1,$ for $n-2 \leq i \leq n,$ or
\item[(ii)] for some $k$ with $1 \leq k \leq n-2$, $\langle \la, \alpha_i^{\vee} \rangle = p-1,$ for $i \in\{k, k+1\},$ and $\langle \la, \alpha_i^{\vee} \rangle = 0,$
for $k+2 \leq i;$
\end{itemize}
\item $\Phi=\rm{C}_{n}$, $n\geq 3$, $p=3,$ and \\$\langle \la, \alpha_i^{\vee} \rangle = p-1,$ for $n-1 \leq i \leq n,$ and $\langle \la, \alpha_{n-2}^{\vee} \rangle \in \{p-2,p-1\};$
\item $\Phi=\rm{D}_{n}$, $n\geq 4$, $p=2,$ and $\langle \la, \alpha_i^{\vee} \rangle = p-1,$ for $n-3 \leq i \leq n$;
\item $\Phi=\rm{E}_{n}$, $n=6,7,8$, $p=2,$ and $\langle \la, \alpha_i^{\vee} \rangle = p-1,$ for $2 \leq i \leq 5$;
\item $\Phi=\rm{F}_{4}$, 
\begin{itemize}
\item[(i)] $p=2$ and  $\langle \la, \alpha_i^{\vee} \rangle = p-1,$ for $1 \leq i \leq 2,$ or
\item[(ii)] $p=3,$ $\langle \la, \alpha_i^{\vee} \rangle = p-1,$ for $2 \leq i \leq 3,$ and $\langle \la, \alpha_4^{\vee} \rangle  \in \{p-2,p-1\};$
\end{itemize}
\item $\Phi=\rm{G}_{2}$, $p=2,$ and $\la =(p-1)\rho.$ 
\end{itemize} 
\end{theorem}

\begin{proof}  For type $\rm{B}_n$ and $p = 2$, when $n = 3$, case (i) follows from  Section \ref{S:TMCB3}, where $\la = \rho = (1,1,1)$.  For $n > 3$, one may consider the Levi subgroup $L_J$ associated to 
$J = \{\a_{n-2},\a_{n-1},\a_{n}\}$. Let $\la = (*,\dots,*,1,1,1)$ be any $p$-restricted weight.  Then the TMC fails for $T_J((p-1)\rho_J + \la)$ over $L_J$. Applying the contrapositive of Theorem \ref{T:TMCLevi}, 
the TMC must fail for $T((p-1)\rho + \la)$ over $G$.   In case (ii), for any $n \geq 3$, the base case of $\la = (1,1,0,\dots,0)$ follows from Theorem \ref{T:Bn} (below).  
For $n > 4$ and $\la = (*,\dots,*,1,1,0,\dots,0)$, where the 1s lie in the $k$th and $(k+1)$st spots (with $k \leq n-2$), 
one uses a Levi subgroup $L_J$ associated to $J = \{\a_k,\a_{k+1},\dots,\a_n\}$.  

For type $\rm{C}_n$ and $p = 3$, when $n = 3$, one base case is $\la = (2,2,2)$, given in Section \ref{S:TMCC3}.  For $n > 3$, the case of $\la = (*,\dots,*,2,2,2)$ follows by using the Levi subgroup $L_J$ associated to $J = \{\a_{n-2},\a_{n-1},\a_n\}$.   The second type $\rm{C}_3$ base case of $\la = (1,2,2)$ is given in Theorem \ref{T:C3-more} (below), from which the general case of $\la = (*,\dots,*,1,2,2)$ similarly follows.

For type $\rm{D}_n$ and $p = 2$, when $n = 4$, the base case of $\la = \rho = (1,1,1,1)$ is given in Section \ref{S:TMCD4}.  For $n > 4$, the case of $\la = (*,\dots,*,1,1,1,1)$ follows using a Levi subgroup $L_J$ associated to $J = \{\a_{n-3},\a_{n-2},\a_{n-1},\a_n\}$.  This type $\rm{D}_4$ case also gives the type $\rm{E}_n$ cases by using a Levi subgroup $L_J$ associated to $J = \{\a_2,\a_3,\a_4,\a_5\}$.

For type $\rm{F}_4$ and $p = 2$, we start with the base case of $\la = (1,1,0)$ for a type $\rm{B}_3$ group, as above.  Using a Levi subgroup $L_J$ associated to $J = \{\a_1,\a_2,\a_3\}$ gives the $\rm{F}_4$ case of $\la = (1,1,*,*)$.  Note that starting with the $\la = (1,1,1)$ case for type $\rm{B}_3$ does not generate any additional examples in type $\rm{F}_4$.   For $p = 3$, we use the two type $\rm{C}_3$ cases of $\la = (2,2,2)$ or $\la = (1,2,2)$ and a Levi subgroup associated to $J = \{\a_2,\a_3,\a_4\}$ to give the type $\rm{F}_4$ cases of $\la = (*,2,2,2)$ or $(*,2,2,1)$, respectively.  Note the ordering swap when changing from type $\rm{C}_3$ to $\rm{F}_4$, as the root $\a_2$ is the long root, whereas $\a_3$, $\a_4$ are the short roots in the type $\rm{C}_3$ subroot system.  

Lastly, the type $\rm{G}_2$ case follows from \cite{BNPS20}.

\end{proof}

\begin{remark} For all pairs $(\Phi, p)$ that appear above  one always has 
$$T(2(p-1)\rho)|_{G_{1}T} \ncong \widehat{Q}_{1}(0).$$
\end{remark}

\subsection{Remarks about the Proof} Our goal was to keep the calculations for 
the low rank groups in Section \ref{S:Low} fairly self-contained. We make use of Ext-data that appears  in the literature, mainly due to Jantzen  \cite{Jan91}, and
 to Dowd and Sin \cite{DS}. In addition, we apply Weyl's character formula, via the computer algebra package  \cite{LiE}, and perform explicit 
  calculations of the multiplicities appearing in the  Jantzen filtration.  Some of these calculations are 
  obtained via a short computer program written for LiE. However, we owe a debt of gratitude to  Frank L\"ubeck for his tables of weight multiplicites  \cite{Lue} and to Stephen
   Doty for his \cite{GAP} package WeylModules \cite{Doty}. These enabled us to look 
   at various examples and observe many of the phenomena that are described in the paper.


\section{More Counterexamples}
\subsection{Type $\rm{B}_n$ Revisited}
With the use of Theorem \ref{T:TMCLevi} and Levi subgroups, we have already observed in Theorem \ref{T:main2} that the TMC fails for type $\rm{B}_n$, $n \geq 3,$, and $p = 2$, based on a counterexample in type $\rm{B}_3$. 
We present here a direct proof of the failure of the TMC for all $n \geq 3$, noting that these counterexamples do not arise from Levi subgroups, thus demonstrating further subtleties in the question of when the TMC holds.   Here we will make use of the construction outlined in Section \ref{S:2nd}.  As observed in Theorem \ref{T:main2}, Levi subgroups may also be used to generate further examples from these examples.  

A key result that will be used involves having some information about the socles of Weyl modules. The proof of the next proposition is technical and will be provided in Section \ref{S:compfactors}.

\begin{prop}\label{P:compfactors} Let $G$ be of type $\rm{B}_n,$ $n \geq 3$, and $p=2.$
Let $\sigma \in X^+$. If $\sigma < \rho -\omega_1$ and $L(\s)$ appears as a composition factor of $\nabla(\rho - \omega_1)$ then $\sigma < \rho - \omega_1 -\omega_2.$
\end{prop}

With Proposition~\ref{P:compfactors}, one can produce new counterexamples to the TMC for type $\rm{B}_n$, $n\geq 3$ and $p=2$. 

\begin{theorem}\label{T:Bn}
Let $G$ be a simple algebraic group of type  $\rm{B}_n$, $n \geq 3,$ and $p=2$. The tilting module $T(\rho+ \omega_1 + \omega_2)$ is not isomorphic to $\widehat{Q}_1(\rho-\omega_1-\omega_2)$ as a $G_1T$-module.
\end{theorem}

\begin{proof}  Assume that the TMC holds for the tilting module $T(\rho+ \omega_1 + \omega_2)$.
Set $\la = \rho -\omega_1-\omega_2$ and $\mu = \rho - \omega_1.$ Note that $\la + 2 \omega_1= \mu + \alpha_1$ and that $\langle \la, \alpha_1^{\vee} \rangle =0.$ 
It follows from Proposition~\ref{P:SecondConstruction2} that 
$$L(\omega_1) \hookrightarrow \nabla (\omega_1) \hookrightarrow \Ext_{G_1}^1(L(\la), \nabla(\mu))^{(-1)}.$$
In addition, Proposition~\ref{P:compfactors} implies that $\Hom_{G_1}(\hQ_1(\la), \nabla(\mu))=0.$ From
Proposition~\ref{P:SecondConstruction} one concludes that 
$$L(\omega_1) \hookrightarrow \Hom_{G_1}(\hQ_1(\la), \hQ_1(\mu))^{(-1)}.$$
Using Theorem~\ref{T:SecondConstruction}, one concludes that 
$L(\omega_1)$ is a submodule of some $\Delta(\ga)$ with $2\ga \leq 2 \rho -(\rho - \omega_1)-(\rho-\omega_1 - \omega_2)= 2 \omega_1 + \omega_2.$ The only possibilities for $\ga$ with $2\omega_1 \leq 2\ga \leq 2\omega_1 + \omega_2$ are $\omega_1$ and $\omega_2.$  Now Lemma~\ref{L:Data} implies
 that the corresponding Weyl modules $\Delta(\omega_1)$, and $\Delta(\omega_2)$ both have the trivial module as their simple socle,  a contradiction. It follows from character data (the only dominant weights of $\nabla(\omega_1)$ are $\omega_1$ and $0$) that $T(\rho + \omega_1)$ is isomorphic to 
 $\widehat{Q}_1(\rho-\omega_1)$ as a $G_{1}T$-module.  From this argument, we can conclude that $T(\rho+ \omega_1 + \omega_2)$ is not isomorphic $\widehat{Q}_1(\rho-\omega_1-\omega_2)$ as a $G_1T$-module.
\end{proof} 

\subsection{} To prove Proposition \ref{P:compfactors}, we will consider the Jantzen Filtration on $\Delta(\rho - \omega_1)$.  
Recall from Section \ref{S:Ja-Fil} that for any $\la \in X^+$
$$
\sum_{i > 0} \operatorname{ch} \Delta(\la)^i = \sum_{\a \in \Phi^+}\sum_{0 < mp < \langle\la + \rho,\a^{\vee}\rangle} \nu_p(mp)\chi(s_{\a,mp}\cdot \la).
$$
The goal is to show that (for $\la = \rho - \omega_1$) most of the $\chi(s_{\a,mp}\cdot\la)$ appearing in the sum in fact vanish.   

We first recall some general facts about the Euler characteristic $\chi(\mu)$ for a weight $\mu$, which is defined as
$$\chi(\mu) = \sum_{\i \ge 0} (-1)^i \text{ch} (R^i \ind_B^G \mu).$$ 
The statements and the proofs can be found in \cite[II 5.4, 5.9, 8.19]{rags}. 

\begin{lemma}\label{L:chifacts} Let $\mu \in X$.  
\begin{itemize}
\item[(a)] If there exists $\a \in \Delta$ with $\langle\mu,\a^{\vee}\rangle = -1$, then $\chi(\mu) = 0$.
\item[(b)] For $w \in W$, $\chi(w\cdot \mu) = (-1)^{\ell(w)}\chi(\mu)$.
\item[(c)] In particular, for $\a \in \Delta$, $\chi(s_{\a}\cdot \mu) = -\chi(\mu)$.
\item[(d)] For $\a \in \Phi^+$, $\chi(s_{\a,mp}\cdot \la) = -\chi(\la - mp\a)$.
\end{itemize}
\end{lemma}

\subsection{} We next identify a condition for vanishing of $\chi(\mu)$ in type $\rm{B}_n$ that will be used repeatedly in Section \ref{S:compfactors}.  The condition is stated in terms of the epsilon-basis for weights. To this point in the paper, per standard convention, weights have been given in the omega-basis, that is, expressing a weight as a linear combination of the fundamental dominant weights $\omega_i$.  By the epsilon-basis, we mean expressing a weight in terms of the standard basis vectors $\epsilon_i$ of the underlying Euclidean space $\mathbb{E}$. For clarity in this section, we will be explicit about which basis is being used.  When weights are written in component notation, let $\mu_{\ep}$ denote the weight $\mu$ in the epsilon-basis and $\mu_{\omega}$ denote the weight the omega-basis. More precisely, for $\mu \in X$, the notation $\mu_{\ep} = (c_1, c_2, \dots, c_n)$ for integers $\{c_s\}$ means $\mu = \sum_{s = 1}^nc_s\ep_s$, whereas $\mu_{\omega} = (c_1, c_2, \dots, c_n)$ for integers $\{c_s\}$ means $\mu = \sum_{s = 1}^nc_s\omega_s$. For example, in type $B_3$, $\rho = \frac52\epsilon_1 + \frac32\epsilon_2 + \frac12\epsilon_3 = \omega_1 + \omega_2 + \omega_3 $, and so we write $\rho_{\epsilon} = (\frac52,\frac32,\frac12)$ and $\rho_{\omega} = (1,1,1)$.  
In the epsilon-basis, the positive roots for $\Phi=\rm{B}_{n}$ consist of $\{\ep_i + \ep_j, 1 \leq i < j \leq n\}$, $\{\ep_i - \ep_j, 1 \leq i < j \leq n\}$, and $\{\ep_i, 1 \leq i \leq n\}$. 

\begin{lemma}\label{L:chivanishing} Assume that the root system $\Phi$ is of type $\rm{B}_n$ or $\rm{C}_n$.   Let $\mu \in X$ with $\mu + \rho = \sum_{s = 1}^nm_s\ep_s$ for integers $\{m_s\}$.  If there exist $1 \leq i < j \leq n$ with $|m_i| = |m_j|$, then $\chi(\mu) = 0$.
\end{lemma}

\begin{proof} Let $\Phi$ be a root system of type $\rm{B}$ or $\rm{C}$. Note that the condition  on $\mu + \rho$ implies that there is a root $\beta \in \Phi$ such that $\langle \mu + \rho, \beta^{\vee} \rangle =0$  (e.g., $\beta = \epsilon_i \pm \epsilon_j$ as appropriate). Using the $W$-invariance of the inner product we can find $w \in W$ and $\alpha \in \Delta$ such that $\langle w \cdot \mu , \alpha^{\vee} \rangle =-1.$ It follows from Lemma \ref{L:chivanishing}(a) and (b) that $\chi(\mu) = 0$.

\end{proof}

\subsection{Proof of Proposition \ref{P:compfactors}}\label{S:compfactors} We are now ready to compute the Jantzen Sum Formula for $\Delta(\rho - \omega_1)$ in type $\rm{B}_n$ and $p = 2$. Proposition \ref{P:compfactors} will follow from the result below, since $\rho-\omega_m - \omega_{m + 1}< \rho-\omega_{1}-\omega_{2}$ for $m\geq 2$. 

\begin{prop}\label{P:Deltastructure}  Let $\Phi$ be of type $\rm{B}_n$ with $n \geq 3$ and $p = 2$.  In the Jantzen Sum Formula, 
$$\sum_{i >0}\operatorname{ch}\Delta(\rho - \omega_1)^i = \sum m_{\mu}\chi(\mu),
$$
for integers $m_{\mu}$, where $\mu = \rho - \omega_m - \omega_{m + 1}$ for $m$ being even and at least 2.  
\end{prop}

\begin{proof} Set $\la = \rho - \omega_1$.  Then $\la + \rho = 2\rho - \omega_1 =  (2n-2)\ep_1 + \sum_{s = 2}^{n}(2(n-s) + 1)\ep_s$ or 
\begin{equation}\label{E:lambda}
(\la + \rho)_{\ep} = (2n-2, 2n-3, 2n-5, \dots, 3, 1). 
\end{equation}
For each positive root $\a$, which will be considered in the epsilon-basis, we consider all $\chi(s_{\a,2m}\cdot\la)$ that may occur.  More precisely, using Lemma \ref{L:chifacts}(d), we consider $\chi(\la - 2m\a)$. For many $\a$, by the nature of $\la - 2m\a + \rho$, we may apply Lemma \ref{L:chivanishing} to conclude that $\chi(\la - 2m\a)$ (and hence $\chi(s_{\a,2m}\cdot\la)$) is zero. In some other cases, we will see that while terms initially survive, they appear more than once and will cancel, in the end leaving only the stated weights.

\vskip.2cm\noindent
{\bf Case 1.} $\a = \ep_i + \ep_j$ for $1 < i < j \leq n$.
\vskip .15cm 
Here $\langle \la + \rho,\a^{\vee}\rangle =  2(n-i) + 2(n-j) + 2$. We show that $\chi(\la - 2m\a) = 0$ for $2 \leq 2m \leq  2(n - i) + 2(n - j)$ by considering the components of $(\la + \rho - 2m\a)_{\ep}$ and applying Lemma \ref{L:chivanishing}.  Write $(\la + \rho - 2m\a)_{\ep} = (c_1, c_2, \dots, c_n)$.  Note that the $c_s$ match those in \eqref{E:lambda} with two exceptions:  $c_i = 2(n-i) + 1 - 2m$ and $c_j = 2(n-j) + 1 - 2m$.  Consider $c_i$. Based on the possible values for $m$, we have
$$
-2(n-(j+1)) - 1 = 2(n-i) + 1 - [2(n-i) + 2(n-j)] \leq c_i \leq 2(n-i) + 1 - 2 = 2(n - (i+1)) + 1.
$$
When $c_i < 0$, we have $-2(n - (j+1)) - 1  \leq c_i \leq -1$.  Therefore $|c_i| = c_t$ for some $j + 1 \leq t \leq n$, and the vanishing follows as claimed.  When $c_i \geq 0$, we have $1 \leq c_i \leq 2(n-(i+1)) + 1$.  Here $c_i = c_t$ for some $i + 1 \leq t \leq n$, with one exception: when $t = j$ or $c_i = 2(n-j) + 1$. That occurs when $2(n-i) + 1 - 2m = 2(n-j) + 1$ or when $2m = 2j - 2i$.  In that situation, 
$$
c_j = 2(n-j) + 1 - 2m = 2(n-j) + 1 - (2j - 2i) = 2(n - 2j + i) + 1.
$$
Since $j > i$, $c_j \neq 2(n-j) + 1$. However, it is possible that $c_j = -2(n-j) - 1$.  In that case, $|c_j| = c_i$, and we are done. 
In general, since $i + 1 \leq j \leq n$, we have
$$ 
-2(n - (i + 1)) - 1 = 2(n - 2n + i) + 1 \leq c_j \leq   2(n - 2(i+1) + i) + 1 = 2(n - (i + 2)) + 1.
$$
If $|c_j| \neq 2(n-j) + 1$, then we see that $|c_j| = c_t$ for some $i + 1 \leq t \leq n$, $t \neq j$, and the claim follows.

\vskip.2cm\noindent
{\bf Case 2.} $\a = \ep_i - \ep_j$ for $1 \leq i < j \leq n$.
\vskip .15cm 
Here 
$$\langle \la + \rho,\a^{\vee}\rangle = 
\begin{cases} 2n - 2 - 2(n - j) - 1 = 2(j - 1) - 1&\text{ if } i = 1\\
 2(n-i) + 1 - 2(n-j) - 1 = 2(j - i) &\text{ if } i > 1.
 \end{cases}
 $$
We need to show that $\chi(\la - 2m\a) = 0$ for $2 \leq 2m \leq  2(j - i) - 2$ (which is vacuous if $j = i + 1$).  We proceed as in Case 1 to apply Lemma \ref{L:chivanishing}.  Write $(\la + \rho - 2m\a)_{\ep} = (c_1, c_2, \dots, c_n)$.  Note that the $c_s$ match those in \eqref{E:lambda} with two exceptions:  $c_i$ and $c_j$.  Here we need to consider only $c_j = 2(n - j) + 1 + 2m$. We have
 $$
 2(n - (j - 1)) + 1 = 2(n-j) + 1 + 2 \leq c_j \leq 2(n-j) + 1 + 2(j-i) - 2 = 2(n - (i + 1)) + 1.
 $$
 Therefore, $c_j = c_t$ for some $i + 1 \leq t \leq j - 1$, and the claim follows. 
 
 \vskip.2cm\noindent
 {\bf Case 3.} $\a = \ep_i$ for $i > 1$.
\vskip .15cm 
Here $\langle \la + \rho,\a^{\vee}\rangle =  2(2(n-i) + 1) = 4(n-i) + 2$. We need to show that $\chi(\la - 2m\a) = 0$ for $2 \leq 2m \leq 4(n-i)$ (which is vacuous if $i = n$).  We proceed as above and apply Lemma \ref{L:chivanishing}.  Write $(\la + \rho - 2m\a)_{\ep} = (c_1, c_2, \dots, c_n)$.  Note that the $c_s$ match those in \eqref{E:lambda} with one exception:  $c_i = 2(n-i) + 1 - 2m$.  Based on the possible values for $m$, we have
$$
-2(n -(i + 1)) - 1 = 2(n-i) + 1 - 4(n-i) \leq c_i \leq 2(n-i) + 1 - 2 = 2(n - (i + 1)) + 1.
$$
Therefore, $|c_i| = c_t$ for some $i + 1 \leq t \leq n$.

 \vskip.2cm\noindent
 {\bf Case 4.} $\a = \ep_1$.
\vskip .15cm 
Here $\langle \la + \rho,\a^{\vee}\rangle =  2(2n - 2) = 4(n-1)$. We need to consider $\chi(s_{\a,2m}\cdot\la) = - \chi(\la - 2m\a)$ for $2 \leq 2m \leq 4(n-1) - 2$.   Unlike the previous cases, these do not all vanish, although some will be seen to cancel.   Consider first the case $2m = 2(n-1)$ (or $m = n - 1$).  By definition, 
$$
\chi(s_{\a,2m}\cdot\la) = \chi(\la - (4(n-1) - 2(n-1))\a) = \chi(\la - 2(n-1)\a) = \chi(\la - 2m\a).
$$
Since this also equals (as noted above) $-\chi(\la - 2m\a)$, we must have $\chi(\la - 2m\a) = 0$.  

The remaining cases will be considered in pairs: $2m$ and $4(n-1) - 2m$ for $2 \leq 2m \leq 2(n - 2)$ or $1 \leq m \leq n - 2$.  
Write $(\la + \rho - 2m\a)_{\ep} = (c_1,c_2,\dots,c_n)$ as before, again, this agrees with \eqref{E:lambda} except in the first component, where $c_1 = 2n - 2 - 2m$.  
On the other hand, the first component of $\la + \rho - [4(n - 1) - 2m]\a$ is 
$$2n - 2 - [4(n-1) - 2m] = -2n + 2 + 2m = -(2n -2 - 2m).$$  
Let $w$ be the reflection in the $\ep_1$-hyperplane, then $w(\la + \rho - 2m\a) = \la + \rho - (4(n-1) - 2m))\a$.  Since the length of $w$ is odd, by Lemma \ref{L:chifacts}(b), $\chi(\la - 2m\a) = -\chi(\la - (4(n-1)-2m)\a)$.  If $m$ is odd, $\nu_2(2m) = 1 = \nu_2(2(2(n-1) - m)) = \nu_2(4(n-1) - 2m)$, and so the two characters cancel.   If $m$ is even, the characters do not necessarily cancel.  

Suppose that $m$ is even and $2 \leq m \leq n - 2$.  Let $w$ be the Weyl group element such that $w\cdot(\la  - 2m\a) = w(\la +\rho - 2m\a) - \rho$ is dominant.  Then $\chi(\la - 2m\a) = (-1)^{\ell(w)}\chi(w\cdot(\la - 2m\a))$. We have
$$
(\la + \rho - 2m\a)_{\ep} = (2n - 2 - 2m, 2n - 3, 2n - 5, \dots, 3, 1).
$$
Note that all components are necessarily positive.  As discussed in the proof of Lemma \ref{L:chivanishing}, we obtain the components of $(w(\la + \rho - 2m\a))_{\ep}$ by placing these in decreasing order (left to right).  So we get
$$
(w(\la + \rho - 2m\a))_{\ep} = (2n - 3, 2n - 5, \dots, 2n - 1 - 2m , 2n - 2 - 2m, 2n - 3 - 2m, \dots, 3, 1).
$$
Hence,
$$
(w(\la + \rho - 2m\a))_{\omega} = (2, 2, \dots, 2,1,1,2,\dots,2,2),
$$
where the $1$'s are in the $m$th and $(m + 1)$st components.  Subtracting $\rho$ gives
$$
(w\cdot(\la - 2m\a))_{\omega} = [w(\la + \rho - 2m\a) - \rho]_{\omega} = (1,\dots,1,0,0,1,\dots, 1),
$$
with the zeros in the same locations as above. Hence, $w\cdot(\la - 2m\a) = \rho - \omega_m - \omega_{m+1}$  and multiples of $\chi(\rho-\omega_m - \omega_{m+1})$ may appear in the sum. 
There are no other contributions in this case.

 \vskip.2cm\noindent
 {\bf Case 5.} $\a = \ep_1 + \ep_j$ for $2 \leq j \leq n$.
 \vskip .15cm 
 Here $\langle \la + \rho,\a^{\vee}\rangle =  2(n - 1) + 2(n - j) + 1$. We need to consider $\chi(s_{\a,2m}\cdot\la) = - \chi(\la - 2m\a)$ for 
 $2 \leq 2m \leq  2(n - 1) + 2(n - j)$. We will use Lemma \ref{L:chivanishing} to show that most of these vanish, with some remaining cases seen to cancel, so that these terms give no further contribution to the sum formula.  Write $(\la + \rho - 2m\a)_{\ep} = (c_1, c_2, \dots, c_n)$.  Note that the $c_s$ match those in \eqref{E:lambda} with two exceptions:  $c_1$ and $c_j$.  Consider $c_j = 2(n-j) + 1 - 2m$.  Based on the values of $m$, we have
 $$
-2(n-2) - 1 =  2(n-j) + 1 - 2(n-1) - 2(n-j) \leq c_j \leq 2(n-j) + 1 - 2 = 2(n - (j+1)) + 1.
 $$
If $c_j \geq 0$, then we have $1 \leq c_j \leq 2(n - (j + 1)) + 1$. So $c_j = c_t$ for some $j + 1 \leq t \leq n$ and the terms vanish by Lemma \ref{L:chivanishing}.   If $c_j < 0$, we have $-2(n-2) - 1 < c_j < -1$ and $|c_j| = c_t$ for some $2 \leq t \leq n$, unless $t = j$.  That is, when $2(n-j) + 1 - 2m = -2(n-j) - 1$ or $2m = 4(n-j) + 2$.

Suppose $2m = 4(n-j) + 2$ and consider the first component: $c_1 = 2n - 2 - 4(n-j) - 2 = -2n + 4j - 4 = -2(n - 2j + 2)$.  If $n$ is even, this is zero when $n = 2j - 2$ or $j = \frac{n+2}{2}$.  Let $w \in W$ be such that $w(\la + \rho - 2m\a)$ is dominant.   Since $(\la + \rho - 2m\a)_{\ep}$ has a zero in the first component, $(w(\la + \rho - 2m\a))_{\ep}$ has a zero in the last component.   Hence, the last component of $(w(\la + \rho - 2m\a))_{\omega}$ is also zero, and so the last component of $(w\cdot(\la - 2m\a))_{\omega} = (w(\la + \rho - 2m\a) - \rho)_{\omega}$ is $-1$.  From Lemma \ref{L:chifacts}(a) and (b), it follows that $\chi(\la - 2m\a) = 0$.  So we are done for $j = \frac{n + 2}{2}$.

Recapping: for any $2 \leq j \leq n$ with $j \neq \frac{n + 2}{2}$, we know that $\chi(\la - 2m\a) = 0$ except when $2m = 4(n-j) + 2$ (or $m = 2(n-j) + 1$).  Split $\{j : 2 \leq j \leq n, j \neq \frac{n+2}{2}\}$  into pairs $j$ and $n + 2 - j$ for $2 \leq j \leq \lfloor \frac{n+1}{2} \rfloor$.  For each such $j$, we show that $\nu_{2}(2m)\chi(\la - 2m\a) = -\nu_{2}(2\tilde{m})\chi(\la - 2\tilde{m}\tilde{\a})$, where $\tilde{m} = 2(j -2) + 1$ is the relevant ``$m$'' for the index $n + 2 - j$ and $\tilde{\a} = \ep_1 + \ep_{n+2 - j}$.  This shows that these remaining characters cancel, so there is no contribution from the $\ep_1 + \ep_j$.

First, observe that both $m$ and $\tilde{m}$ are odd, so $\nu_2(2m) = 1 = \nu_2(2\tilde{m})$.  As above, write $(\la + \rho - 2m\a)_{\ep} = (c_1,c_2, \dots,c_n)$. Then the $c_s$ match those in \eqref{E:lambda} with two exceptions: 
$$
c_1 = -2(n - 2j + 2) \text{ and } c_j = 2(n-j) + 1 - 4(n-j) - 2 = -2(n - j) - 1.
$$  
Similarly, write $(\la + \rho -2\tilde{m}\tilde{\a})_{\ep} = (d_1, d_2, \dots,d_n)$. Again, the $d_s$ agree with those in \eqref{E:lambda} with two exceptions:
$$
 d_1 = 2n - 2 - 4(j-2) - 2 = 2(n-2j + 2) \text{ and } d_{n + 2 - j} = 2(j - 2) + 1 - 4(j-2) - 2 = -2(j-2) - 1.
 $$
We see that $c_s = d_s$ for $s \neq 1, j, n + 2 - j$, while $c_1 = -d_1$, $c_j = -d_j$, and $c_{n+2-j} = -d_{n+2-j}$.  Let $w_i$ denote the reflection in the $\ep_i$-hyperplane and set $w = w_1w_jw_{n+2-j}$ (the ordering being irrelevant).  Then $w(\la + \rho - 2m\a) = \la + \rho - 2\tilde{m}\tilde{\a}$, and so $\chi(\la - 2m\a) = (-1)^{\ell(w)}\chi(w\cdot(\la - 2m\a) = (-1)^{\ell(w)}\chi(\la - 2\tilde{m}\tilde{\a})$.  Since each $w_i$ has odd length, so does $w$, and the claim that $\nu_{2}(2m)\chi(\la - 2m\a) = -\nu_{2}(2\tilde{m})\chi(\la - 2\tilde{m}\tilde{\a})$ follows.  

\end{proof}

\subsection{A Second Failure of the TMC for Type $\rm{C}_3$}\label{S:2ndC3}
Let $G$ be of type $\rm{C}_3$ with $p = 3$. In Section \ref{S:TMCC3}, we saw that $T(4\rho)$ fails to be indecomposable upon restriction to $G_1T$. In this section, we will show that the tilting module $T(4\rho-\omega_1)$ also fails to remain indecomposable when restricted to $G_1T.$ Note that, in this subsection, we return to expressing weights solely in the omega-basis.

 \begin{theorem}\label{T:C3-more} Let $G$ be a simple algebraic group of type $\rm{C}_3$ and $p=3.$ The tilting module $T(4 \rho-\omega_1) $ is not isomorphic to $\widehat{Q}_1(\omega_1)$ as a $G_1T$-module. 
\end{theorem}

\noindent 
Data obtained via the Jantzen Filtration yields the following table:

\vskip.25cm
\begin{center}
\begin{tabular}{|c|c|}
\hline
$\la$ & $\sum_{i>0} \text{ch } \Delta(\la)^i$\\
\hline
\hline
$(1,0,0)$ & $\emptyset$ \\
\hline
$(0,1,1)$ & $\emptyset$ \\
\hline
$(2,1,1)$ &  $\chi(0,1,1)$ \\
\hline
$(1,3,0)$ & $\chi(2,1,1))- \chi(0,1,1)+ \chi(1,0,0)$ \\
\hline
$(3,2,0)$ &  $\chi(1,3,0) + \chi(2,1,1))+ \chi(1,0,0)$ \\
\hline
 $(2,2,1)$& $\chi(3,2,0) + 2 \cdot \chi(1,3,0) + \chi(2,1,1) - 2\cdot \chi(1,0,0)$ \\
 \hline
 \end{tabular}
\end{center}
\vskip.25cm

It follows that $\Delta(1,0,0)$ and $\Delta(0,1,1)$ are simple and that $\Delta(2,1,1)$ has length two, the second factor having highest weight $(0,1,1).$ Character considerations now show that $\Delta(1,3,0)$ is multiplicity free with three composition factors, including factors with highest weights $(2,1,1)$ and $(1,0,0).$ Similarly one concludes from character data that $\Delta(3,2,0)$ has six composition factors, namely
the simple modules with highest weights $(3,2,0)$, $(1,3,0)$, $(2,1,1)$, $(0,1,1)$, and $(1,0,0)$ twice.
One concludes that  
$$\sum_{i>0} \text{ch } \Delta(2,2,1)^i=  \text{ch }L(3,2,0) + 2 \cdot \text{ch }L(0,1,1) + 2 \cdot \text{ch }L(1,0,0) + 3 \cdot \text{ch }L(1,3,0) + 4 \cdot \text{ch }L(2,1,1).$$

Suppose that $T(4\rho - \omega_1) \cong \widehat{Q}_1(\omega_1)$ as $G_{1}T$-module.  Then 
$$V := \Hom_{G_1}(\hQ_1(1,0,0),\nabla(2,2,1))^{(-1)}$$ 
admits a good filtration. 
We argue in a similar manner as in Section \ref{S:TMCC3}.  
From above, we see that $V$ has composition factors of $L(\omega_2)$ and $k$, with the number (and arrangement) of such factors determined by the the number (and arrangment) of copies of $L(1,3,0)$ and $L(1,0,0)$, respectively, in $\nabla(2,2,1)$.   From above, we have that $\Delta(2,2,1)$ (and hence $\nabla (2,2,1)$) contains between one and three copies of $L(1,3,0)$  and one or  two copies of $L(1,0,0)$.  
Analogous to the argument in Case 1 of Section \ref{S:TMCC3}, for $V$ to admit a good filtration, the number of  copies of $L(1,0,0)$ appearing must be at least the number of copies of $L(1,3,0)$ that appear. In particular, there can in fact be at most 2 copies of $L(1,3,0)$.   Suppose that $[\Delta(2,2,1) : L(1,0,0)] = 2$.  Then $[\Delta(2,2,1)^2 : L(1,0,0)] = 0$. 
Furthermore, $[\Delta(2,2,1) : L(1,3,0)] = 1 \text{ or } 2$, and, in either case, $[\Delta(2,2,1)^2 : L(1,3,0)] = 1$.  One may now argue as in Case 2 of Section \ref{S:TMCC3}, starting from an exact sequence
$$
0 \to \Delta(2,2,1)^2 \to \Delta(2,2,1) \to Q \to 0,
$$
to obtain a similar contradiction.
Lastly, suppose $[\Delta(2,2,1) : L(1,0,0)] = 1$, then 
\begin{align*}
[\Delta(2,2,1)^2 : L(1,0,0)] &= 1,\\
[\Delta(2,2,1) : L(1,3,0)] &= 1,\\
[\Delta(2,2,1)^2 : L(1,3,0)] &= 0, \text{ and }\\
[\Delta(2,2,1)^3 : L(1,3,0)] &= 1.
\end{align*}
In this case, one starts from an exact sequence
$$
0 \to \Delta(2,2,1)^3 \to \Delta(2,2,1) \to Q \to 0
$$
to obtain a contradiction.


\section{Further Questions}\label{S:questions}

\subsection{}

We define the condition (ET) as

\vspace{0.2in}

\noindent (ET) \quad $\Ext_{G_r}^1(L(\lambda),L(\mu))^{(-r)}$ \; is tilting as a $G$-module for all \; $\lambda,\mu \in X_r$.
\vspace{0.2in}

The evidence thus far suggests that there is some connection between this condition and the TMC.  Indeed, (ET) fails in each of the (known) low rank counterexamples to the TMC, and in the present paper, the former served to help detect the latter.  At the same time, in every case in which we know that the TMC holds, we also know that (ET) holds.

In the course of this discussion, one would like to know the answer to the following question, which is of interest in its own right.

\begin{quest}
    Under what conditions does (ET) hold?
\end{quest}

In \cite[Thm. 4.3.1]{BNPS23}, the authors proved that if $p \ge 2h-4$, then (ET) holds for $r = 1$. Moreover, the structure is semi-simple (so that the tilting factors are all simple $G$-modules). This improved an earlier confirmation of (ET) for $r = 1$ by Andersen \cite{And84} for $p \ge 3h-3$ and by Bendel, Nakano, and Pillen \cite[\S 5.5]{BNP04} for $p \geq 2h - 2$.  In the recent work \cite{BNPS23} of the authors, it was also shown that the TMC holds when $p\geq 2h-4$. Further sharpening of this bound in (ET) could be a key step in lowering the bound of validity for the TMC.

\subsection{}

Another mystery to unravel is precisely how (ET) and the TMC relate to each other.

\begin{quest}
    Is the TMC equivalent to (ET)?  If not, does one imply the other?
\end{quest}

A revised problem is to consider this for each restricted $\lambda$.

\begin{quest}
    Is $T(2(p^r-1)\rho+w_0\lambda)$ indecomposable over $G_r$ if and only if \\
    $\Ext_{G_r}^1(L(\lambda),L(\mu))^{(-r)}$ is tilting for all $\mu \in X_r$ where $\mu \ge_{\mathbb{Q}} \lambda$?
\end{quest}

To motivate this question, as observed in Section \ref{S:1st}, if $Q_r(\la)$ and $Q_r(\mu)$ admit a $G$-structure, then we have an embedding of $G$-modules
$$\Ext_{G_r}^1(L(\lambda),L(\mu))^{(-r)} \subseteq \Hom_{G_r}(Q_r(\lambda),Q_r(\mu))^{(-r)}.$$
The idea then is that if the submodule is not tilting over $G$, potentially the larger Hom-set cannot carry such a structure either, in which case one or both of $Q_r(\lambda)$ and $Q_r(\mu)$ cannot lift to a tilting module for $G$.

The final stipulation, that $\mu \ge_{\mathbb{Q}} \lambda$, is to isolate the larger of the two projective covers, $Q_r(\lambda)$, as the one that must fail to be tilting if only one of two does.

\subsection{}

We also have observed that the $G_1T$-projective cover of the trivial module fails to be isomorphic to $T(2(p-1)\rho)$ in all cases in which the TMC fails.  This raises another question as to whether the TMC is equivalent 
to $T(2(p-1)\rho)\mid_{G_{1}T}\cong \widehat{Q}(0)$. 

\begin{quest}
    Does the TMC hold if and only if $T(2(p-1)\rho)$ is indecomposable over $G_1$?
\end{quest}

We note that if there is an affirmative answer to this question and to the questions raised in the previous subsections, then in such a case the problem of checking the validity of the TMC would reduce to checking if
$\text{H}^1(G_1,L(\lambda))^{(-1)}$ is tilting for all $\lambda \in X_1$.

\providecommand{\bysame}{\leavevmode\hbox
to3em{\hrulefill}\thinspace}

\end{document}